\documentclass{amsart}
\title{Normed symmetric monoidal categories}
\author{Jonathan Rubin}
\date{\today}
\usepackage{amsmath,amssymb,amsthm,stackrel,stmaryrd,tikz,enumitem,mathrsfs,physics,graphicx}

\newcommand{\s}[1]{\mathscr{#1}}
\newcommand{\tn}{\otimes}
\newcommand{\btn}{\bigotimes}
\newcommand{\bsq}{\bigsqcup\!{}}
\newcommand{\btnt}{\!\!\! \begin{array}{c} \bigotimes_T \end{array} \!\!\!}
\newcommand{\sbtn}{\!\!\! \begin{array}{c} \bigotimes \end{array} \!\!\!}
\newcommand{\ul}[1]{\underline{#1}}
\newcommand{\ol}[1]{\overline{#1}}
\newcommand{\ub}[1]{\underline{\mathbf{#1}}}

\renewcommand{\b}[1]{\mathbf{#1}}
\newcommand{\bb}[1]{\mathbb{#1}}
\renewcommand{\c}[1]{\mathcal{#1}}
\newcommand{\til}[1]{\widetilde{#1}}
\renewcommand{\t}[1]{\textnormal{#1}}
\newcommand{\vc}{\bullet}

\newcommand{\ups}{\upsilon}

\newcommand{\ve}{\varepsilon}
\newcommand{\vp}{\varphi}

\newcommand{\vd}{\partial}

\newcommand{\id}{\textnormal{id}}

\newcommand{\la}{\langle}
\newcommand{\ra}{\rangle}

\theoremstyle{plain}
\newtheorem*{thm*}{Theorem}
\newtheorem{thm}{Theorem}[section]

\newtheorem{prop}[thm]{Proposition}
\newtheorem{cor}[thm]{Corollary}
\newtheorem*{cor*}{Corollary}
\newtheorem{lem}[thm]{Lemma}

\newtheorem*{thm1}{Theorem \protect{\ref{thm:coh}}}
\newtheorem*{thm2}{Theorem \protect{\ref{thm:NSMCmodel}}}
\newtheorem*{thm3}{Proposition \protect{\ref{prop:TGCnorms}}}

\theoremstyle{definition}
\newtheorem{defn}[thm]{Definition}
\newtheorem{ex}[thm]{Example}
\newtheorem*{ex*}{Example}
\newtheorem{nonex}[thm]{Nonexample}
\newtheorem{nota}[thm]{Notation}

\newtheorem{conv}[thm]{Convention}
\newtheorem{const}[thm]{Construction}

\theoremstyle{remark}
\newtheorem{rem}[thm]{Remark}
\newtheorem*{rem*}{Remark}

\begin{document}

\maketitle

\begin{abstract} We introduce categorical models of $N_\infty$ spaces, which we call normed symmetric monoidal categories (NSMCs). These are ordinary symmetric monoidal categories equipped with compatible families of norm maps, and when specialized to a particular class of examples, they reveal a connection between the equivariant symmetric monoidal categories of Guillou--May--Merling--Osorno and those of Hill--Hopkins. We also give an operadic interpretation of the Mac Lane coherence theorem and generalize it to include NSMCs. Among other things, this theorem ensures that the classifying space of a NSMC is a $N_\infty$ space. We conclude by extending our coherence theorem to include NSMCs with strict relations. 
\end{abstract}

\tableofcontents

\section{Introduction}

Although category theory is often regarded as a language for describing mathematics, one can also study categories as homotopical objects in their own right. From the latter perspective, a category is a simultaneous generalization of a poset and a monoid, and as explained by Segal \cite{Segalclass},  every small category $\s{C}$ gives rise to a classifying space $B\s{C}$ that encodes the relations in $\s{C}$. This is a CW complex with a vertex for each object of $\s{C}$, an edge for each morphism of $\s{C}$, a $2$-cell for each composable pair of morphisms in $\s{C}$, an so on.

The classifying space functor has a number of pleasant properties, and extra structure on a category frequently gives rise to extra structure on its classifying space. In particular, if $\s{C}$ is a symmetric monoidal category, i.e. a category equipped with an operation that is associative, unital, and commutative up to coherent natural isomorphism, then by work of Segal \cite{Segal} and May \cite{MayGILS}, \cite{MayPerm}, its classifying space $B\s{C}$ is an $E_\infty$ space, i.e. a space that is equipped with an operation that is associative, unital, and commutative up to coherent homotopy. 

That being said, specifying an $E_\infty$ structure on a space generally requires far more data than specifying a symmetric monoidal structure on a category. From an operadic standpoint, an $E_\infty$ structure on a space $X$ consists of a compatible family of structure maps $E\Sigma_n \times X^{\times n} \to X$ for all $n \geq 0$. The vertices of $E\Sigma_n$ correspond to $n$-ary operations, the cells of $E\Sigma_n$ correspond to (higher) homotopies between $n$-ary operations, and since $E\Sigma_n$ has infinitely many cells when $n > 1$, an $E_\infty$ structure on $X$ consists of infinitely many operations and infinitely many coherence relations between them. On the other hand, a symmetric monoidal structure on a category $\s{C}$ consists of a product $\otimes$, a unit object $e$, and four natural isomorphisms $\alpha$, $\lambda$, $\rho$, and $\beta$ that witness the associativity, unitality, and commutativity of $\otimes$. These data are required to make four diagrams commute, and from here the Mac Lane coherence theorem \cite{MacLane}, along with Kelly's simplifications \cite{Kelly}, guarantees that all sensible diagrams built from $\alpha$, $\lambda$, $\rho$, and $\beta$ commute. Crucially, the categorical coherence enjoyed by a symmetric monoidal product on $\s{C}$ gives rise to the homotopy coherence required of an $E_\infty$ structure on $B\s{C}$.

In this paper, we study the equivariant analogue to this situation. Let $G$ be a finite group. If $\s{C}$ is a category with a $G$-action, then $B\s{C}$ is a space with a $G$-action by functoriality. That being said, the story becomes more complicated once we introduce algebraic structure, because there are multiple notions of equivariant homotopy commutativity. By way of motivation, suppose $(A,*)$ is a commutative monoid equipped with a $G$-action, and consider the fixed points of $A$. Then for any subgroup $H \subset G$, the fixed set $A^H$ inherits a commutative monoid structure from $A$.  More interestingly, for any subgroups $K \subset H$, there is a wrong-way ``norm map'' $\t{n}_K^H : A^K \to A^H$, which sends $x \in A^K$ to the product $(r_1 x) * \cdots * (r_n x)$, where $r_1,\dots,r_n$ are a set of $H/K$ coset representatives. That $\t{n}_K^H(x)$ is $H$-fixed follows from the strict commutativity of the product $*$.

Now suppose that the product $*$ is only homotopy commutative. In this case, norm maps $A^K \to A^H$ can and often do exist, but they are additional structure. Accordingly, there are many different flavors of equivariant homotopy commutativity -- not only must one specify a product, but also which norms are present. The situation can be made precise using Blumberg and Hill's theory of $N_\infty$ algebras \cite{BH}. Roughly speaking, a $N_\infty$ algebra is an equivariant space, which is equipped with a homotopy coherent commutative monoid structure, together with a compatible system of norm maps.

Given the connection between $E_\infty$ spaces and symmetric monoidal categories, it is natural to ask what the categorical analogue to a $N_\infty$ space is. In what follows, we introduce an equivariant generalization of ordinary symmetric monoidal categories, which we call \emph{normed symmetric monoidal categories} (NSMCs). These are symmetric monoidal categories equipped with compatible families of norm maps. We shall prove a coherence theorem for NSMCs (Theorem \ref{thm:coh}), which informally states that every diagram built from certain basic isomorphisms in a NSMC commutes. In addition to its purely algebraic consequences, this theorem implies that the classifying space of a NSMC is a $N_\infty$ space (Theorem \ref{thm:NSMCmodel}).

As usual, making a coherence theorem precise requires work, because one must restrict attention to a certain class of formally definable diagrams. We shall use a blend of first-order logical and operadic formalism. Given any set $\c{T}$ of norm maps, we consider a certain categorical operad $\c{SM}_{\c{T}}$, whose objects are formal composites of the ordinary product and the norms in $\c{T}$, and which has a unique isomorphism between any two composites of the same arity. We then prove that $\c{T}$-normed symmetric monoidal categories are equivalent to $\c{SM}_{\c{T}}$-algebras. More precisely, we generalize and recast Mac Lane's coherence theorem as follows.

\begin{thm1} The $2$-category of all $\c{T}$-NSMCs, lax $\c{T}$-monoidal functors, and $\c{T}$-transformations is isomorphic to the $2$-category of all $\c{SM}_{\c{T}}$-algebras, lax $\c{SM}_{\c{T}}$-morphisms, and $\c{SM}_{\c{T}}$-transformations. Similarly for strong and strict morphisms.
\end{thm1}

When $G$ is trivial and $\c{T}$ is empty, this is an operadic reinterpretation of Mac Lane's classical work, but it is a new result in all other cases. We generalize Theorem \ref{thm:coh} to include NSMCs with strict relations in Theorem \ref{thm:cohR}.

The approach to coherence in Theorem \ref{thm:coh} has several benefits. To start, its formulation as an isomorphism of $2$-categories allows for simultaneous consideration of coherence for categories, functors, and natural transformations. Next, the description of $\c{SM}_{\c{T}}$ in terms of formal composite operations links our work to Mac Lane's classical coherence theory. This allows us to use Mac Lane's results and to adapt his proofs in \cite{MacLane} and \cite{CWM}. Finally, the operadic description of coherence makes for a smooth transition to topology. As observed by May \cite{MayPerm}, the classifying space functor $B$ preserves operad actions, and with a little bit of work, the next result follows.

\begin{thm2}  Let $\s{C}$ be a $\c{T}$-NSMC. Then the classifying space $B\s{C}$ is an algebra over the $N_\infty$ operad $B\c{SM}_{\c{T}}$. {Moreover, the class of admissible sets of $B\c{SM}_{\c{T}}$ is the indexing system\footnote{This is a combinatorial object that encodes the norms parametrized by a $N_\infty$ operad.} generated by $\c{T}$.}
\end{thm2}

On the other hand, our theory of NSMCs also has applications in equivariant category theory, and it partially bridges the gap between the frameworks developed by Guillou--May--Merling--Osorno \cite{GM}, \cite{GMM}, \cite{GMMO} and Hill--Hopkins \cite{HH}. These are two distinct approaches to $G$-equivariant symmetric monoidal structure, with two distinct aims. Guillou--May--Merling--Osorno consider (pseudo)algebras over an operad $\s{P}_G$, which is a $G$-equivariantization of the categorical Barratt--Eccles operad. These equivariant symmetric monoidal categories are regarded as models for $E_\infty$ $G$-spaces, and they are input to the authors' infinite loop space machinery. On the other hand, Hill and Hopkins consider $G$-coefficient systems of symmetric monoidal categories, which are equipped with operations analogous to the indexed monoidal products of \cite{HHR}. These equivariant symmetric monoidal categories provide a language for describing phenomena related to the transfer.

A priori, we see no reason why Guillou--May--Merling--Osorno's and Hill--Hopkins' categories should be related in any compelling way, but here is where NSMCs come in. Suppose $\s{C}$ is a strictly associative and unital symmetric monoidal category, and that $J$ is a small right $G$-category. Then the category $\b{Fun}(J,\s{C})$ of all functors from $J$ to $\s{C}$ is a NSMC (cf. \S\ref{sec:FunNSMC}). Now specialize to the case where $J = \bb{T}G$ is the translation category of $G$, i.e. the category whose object set is $G$ and which has a unique morphism between any pair of objects. Then the category $\b{Fun}(\bb{T}G,\s{C})$ is a prototypical example of a $\s{P}_G$-algebra, and we can identify its fixed points.

\begin{thm3} Let $\s{C}$ be an ordinary symmetric monoidal category and let $\bb{T}G$ be the translation category of $G$. Then the fixed points of $\b{Fun}(\bb{T}G,\s{C})$ are equivalent to the categorical Mackey functor $\underline{\s{C}}^{(-)}$, whose value at $G/H$ is the category of $H$-actions in $\s{C}$, and whose transfers are Hill--Hopkins--Ravenel norms.
\end{thm3}

Composing the restrictions and transfers of $\underline{\s{C}}^{(-)}$ yields a prototypical example of a $G$-symmetric monoidal category in the sense of Hill and Hopkins. In this way, NSMCs relate a standard class of $\s{P}_G$-algebras to a standard class of $G$-symmetric monoidal categories. Moreover, Theorem \ref{thm:coh} describes the coherence enjoyed by the category $\b{Fun}(\bb{T}G,\s{C})$, and by extension, the coherence of the categorical Mackey functor $\ul{\s{C}}^{(-)}$ and its associated $G$-symmetric monoidal category.

Thus, we take two perspectives in this paper. One is purely categorical: we introduce an equivariant generalization of symmetric monoidal structure, which in some cases interpolates between Guillou--May--Merling--Osorno's and Hill--Hopkins' equivariant categories, and we prove a version of the Mac Lane coherence theorem for our structures. The vast majority of this paper is written from this point of view. The second perspective is homotopical. By way of our coherence theorem, we deduce that the classifying space of a NSMC is a $N_\infty$ space, and thus obtain a means of constructing $N_\infty$ spaces from a small amount of categorical data.

\subsection*{Organization} In \S\ref{ch:NSMC,sec:defn}, we define NSMCs and the $2$-categories that they assemble into, and then we analyze functor NSMCs in \S\ref{sec:FunNSMC}. In \S\ref{sec:opalg2cats}, we define algebras over categorical operads and the $2$-categories that they assemble into. In \S\ref{ch:NSMC,sec:cohthm}, we state the coherence theorem for NSMCs, and in \S\ref{sec:pfcoh} we prove it. We conclude by analyzing NSMCs with strict relations in \S\ref{sec:TRNSMC}.

\subsection*{Conventions} Throughout this paper, $G$ denotes a finite group, and Roman letters such as $K$ and $H$ denote subgroups of $G$. Roman letters around $T$ are actions by subgroups of $G$. Script letters around $\s{C}$ and $\s{D}$ denote categories, and we shall almost always write $\otimes$ for symmetric monoidal products. Script letters around $\s{O}$ and $\s{P}$ denote operads, and all operads are understood to be symmetric. 

\subsection*{Acknowledgements} We want to thank Peter May, Mike Hill, Bert Guillou, Luis Pereira, and Peter Bonventre for many helpful conversations and for their comments on this paper and its previous incarnations. This work was partially supported by NSF grant DMS-1803426 (to the author) and NSF grant 1553653 (to Maryanthe Malliaris).

\section{$2$-categories of normed symmetric monoidal categories}\label{ch:NSMC,sec:defn}

\subsection{Preliminaries on $G$-categories} Fix a finite group $G$. A \emph{$G$-category} is a category $\s{C}$, equipped with an action of $G$ through functors $g \cdot (-) : \s{C} \to \s{C}$. We assume that this action is strict, i.e. that $e \cdot (-) = \t{id}$ and $[h \cdot (-)] \circ [g \cdot (-)] = (hg) \cdot (-)$. If $\s{C}$ and $\s{D}$ are $G$-categories, then a $G$-functor $F : \s{C} \to \s{D}$ is a functor such that $gF(-) = Fg(-)$ for all $g \in G$. If $F,H : \s{C} \rightrightarrows \s{D}$ are $G$-functors, then a \emph{$G$-natural transformation} $\eta : F \Rightarrow H$ is a natural transformation such that $g\eta_{(-)} = \eta_{g(-)}$ for all $g \in G$. Let $\ub{Cat}^G$ denote the $2$-category of all small $G$-categories, $G$-functors, and $G$-natural transformations, and let $\b{Cat}^G$ denote its underlying $1$-category.

The $1$-category $\b{Cat}^G$ is cartesian closed. For any $G$-categories $\s{C}$ and $\s{D}$, the internal hom object $\b{Cat}_G(\s{C},\s{D}) \in \b{Cat}^G$ is the category $\b{Fun}(\s{C},\s{D})$ of \emph{all} (possibly nonequivariant) functors $\s{C} \to \s{D}$, together with all natural transformations between them. The group $G$ acts by conjugation, and the $G$-fixed subcategory of $\b{Cat}_G(\s{C},\s{D})$ is the $1$-category $\ub{Cat}^G(\s{C},\s{D})$ of all $G$-functors $\s{C} \to \s{D}$ and $G$-natural transformations between them.

The adjunction $(-) \times \s{C} \dashv \b{Cat}_G(\s{C},-)$ makes the hom categories $\b{Cat}_G(\s{C},\s{D})$ natural and ubiquitous, especially for our purposes. Consequently, we will be confronted with nonequivariant functors and natural transformations throughout this discussion, even though we are ultimately interested in $G$-equivariant symmetric monoidal categories.

\subsection{Normed symmetric monoidal categories} In this section, we introduce NSMCs. We begin by considering the most direct $G$-equivariant generalization of symmetric monoidal structure.

\begin{defn}\label{defn:symmmonobj} A \emph{symmetric monoidal object in $\b{Cat}^G$} is a tuple $(\s{C},\tn,e,\alpha,\lambda,\rho,\beta)$ such that $\s{C}$ is a small $G$-category, $\tn : \s{C}^{\times 2} \to \s{C}$ is a $G$-bifunctor, $e \in \s{C}$ is a $G$-fixed object, and
	\[
	(C \tn D) \tn E \stackrel{\alpha_{C,D,E}}{\to} C \tn (D \tn E)	,\quad	e \tn C \stackrel{\lambda_C}{\to} C		,\quad	C \tn e \stackrel{\rho_C}{\to} C	,\quad	C \tn D \stackrel{\beta_{C,D}}{\to} D \tn C
	\]
are $G$-natural isomorphisms that make the usual associativity pentagon, braid hexagon, and triangle diagrams commute (cf. \cite{Kelly}). Define the \emph{standard $n$-fold tensor products} on $\s{C}$ by $\btn_0() := e$, $\btn_1(C) = C$, and $\btn_{n+1}(C_1,\dots,C_{n+1}) = \btn_n(C_1,\dots,C_n) \tn C_{n+1}$.
\end{defn}

These data make the fixed point subcategories $\s{C}^H$ into ordinary symmetric monoidal categories, but in order to get transfer maps, we must overlay certain ``twisted equivariant'' tensor products. These twisted tensor products generalize and lift the usual norms $\t{n}_K^H : \s{C}^K \to \s{C}^H$ to operations defined on all of $\s{C}$ (cf. Definition \ref{defn:internalnorm}), so we refer to them as external norms.

\begin{defn}\label{def:extnorm} Suppose that $H \subset G$ is a subgroup and that $T$ is a finite $H$-set equipped with a linear order $T \cong \{1,\dots,|T|\}$. Let $\sigma : H \to \Sigma_{|T|}$ denote the corresponding permutation representation on $\{1,\dots,|T|\}$. For any $G$-category $\s{C}$, we define $\s{C}^{\times T}$ to be the $|T|$-fold cartesian power $\s{C}^{\times |T|}$ equipped with the $H$-action $h(C_1,\dots,C_{|T|}) = (hC_{\sigma(h)^{-1}1} , \dots , hC_{\sigma(h)^{-1}|T|})$, and similarly for morphisms. We define an \emph{external $T$-norm} on $\s{C}$ to be an $H$-functor $\s{C}^{\times T} \to \s{C}$.
\end{defn}

A normed symmetric monoidal category is a symmetric monoidal object equipped with a compatible family of external norms.

\begin{defn}\label{defn:NSMC} A \emph{set of exponents} is an indexed set $\c{T} = (T_i)_{i \in I}$ such that for each $i \in I$, the element $T_i$ is a finite, ordered $H_i$-set for some subgroup $H_i \subset G$ depending on $i$.

A \emph{$\c{T}$-normed symmetric monoidal category} ($\c{T}$-NSMC) is a symmetric monoidal object $(\s{C},\tn,e,\alpha,\lambda,\rho,\beta)$ in $\b{Cat}^G$, together with
	\begin{enumerate}
		\item{}an external $T_i$-norm $\btn_{T_i} : \s{C}^{\times T_i} \to \s{C}$ for every $i \in I$, and
		\item{}(\emph{untwistors}) a possibly nonequivariant natural isomorphism
		\[
		\ups_{T_i} : \sbtn_{T_i} (C_1,\dots,C_{|T_i|}) \to \sbtn_{|T_i|}(C_1,\dots,C_{|T_i|})
		\]
		for every $i \in I$, such that for every $h \in H_i$, the ``twisted equivariance'' diagram
		\[
		\begin{tikzpicture}
			\node(A) at (0,0) {$h\sbtn_{T_i}(C_1,\dots,C_{|T_i|})$};
			\node(B) at (5,0) {$\sbtn_{T_i}(hC_{\sigma(h)^{-1}1},\dots,hC_{\sigma(h)^{-1}|T_i|})$};
			\node(C) at (5,-2) {$\sbtn_{|T_i|}(hC_{\sigma(h)^{-1}1},\dots,hC_{\sigma(h)^{-1}|T_i|})$};
			\node(D) at (0,-4) {$h\sbtn_{|T_i|}(C_1,\dots,C_{|T_i|})$};
			\node(E) at (5,-4) {$\sbtn_{|T_i|}(hC_1,\dots,hC_{|T_i|})$};
			\path[->]
			(A) edge [above] node {$\id$} (B)
			(A) edge [left] node {$h\ups_{T_i}$} (D)
			(B) edge [right] node {$\ups_{T_i}$} (C)
			(C) edge [right] node {$\sigma(h)^{-1}$} (E)
			(D) edge [below] node {$\id$} (E)
			;
		\end{tikzpicture}
		\]
		commutes. Here the morphism $\sigma(h)^{-1}$ denotes the canonical isomorphism for the symmetric monoidal category $\s{C}$, which permutes the factors of $\btn_{|T_i|}$ by $\sigma(h)^{-1}$.
	\end{enumerate}
\end{defn}

\begin{nota} We shall write $\tn^{\s{C}}$, $\alpha^{\s{C}}$, $\lambda^{\s{C}}$, etc. when we wish to emphasize that these data are associated to a particular NSMC $\s{C}$.
\end{nota}

Just as categories are the objects of a $2$-category, so too are symmetric monoidal categories and NSMCs. We now generalize the usual notions of (lax, strong, and strict) monoidal functors and monoidal natural transformations to the normed symmetric monoidal setting.

\begin{defn}\label{defn:laxnormedfunctor} Suppose that $\s{C}$ and $\s{D}$ are $\c{T}$-normed symmetric monoidal categories. A \emph{lax $\c{T}$-normed functor} $(F,F_\bullet) : \s{C} \to \s{D}$ consists of the following data:
	\begin{enumerate}
		\item{}a $G$-functor $F : \s{C} \to \s{D}$,
		\item{}a $G$-fixed morphism $F_e : e^{\s{D}} \to Fe^{\s{C}}$,
		\item{}a $G$-natural transformation $F_\tn : FC \tn^{\s{D}} FC' \to F(C \tn^{\s{C}} C')$, and
		\item{}for every $i \in I$, a $H_i$-natural transformation\footnote{Regard both sides as $H_i$-functors $\s{C}^{\times T_i} \to \s{D}$.} 
			\[
			F_{\btn_{T_i}} : \sbtn^{\s{D}}_{T_i}(FC_1 , \dots , FC_{|T_i|}) \to F\big(\sbtn^{\s{C}}_{T_i}(C_1 , \dots , C_{|T_i|}) \big),
			\]
	\end{enumerate}
such that the usual lax symmetric monoidal diagrams relating $\alpha$, $\lambda$, $\rho$, and $\beta$ to the comparison maps $F_e$ and $F_\tn$ commute (cf. \cite[Ch. XI.2]{CWM}), and the square
	\[
	\begin{tikzpicture}[scale=0.85]
	\node(ul) {$\btn^{\s{D}}_{T_i}(FC_1 , \dots , FC_{|T_i|})$};
	\node(ur) at (6.5,0) {$\btn^{\s{D}}_{|T_i|}(FC_1 , \dots , FC_{|T_i|})$};
	\node(ll) at (0,-3) {$F\Big( \btn^{\s{C}}_{T_i}(C_1 , \dots , C_{|T_i|}) \Big)$};
	\node(lr) at (6.5,-3) {$F\Big( \btn^{\s{C}}_{|T_i|}(C_1 , \dots , C_{|T_i|}) \Big)$};
	\path[->]
	(ul) edge [left] node {$F_{\btn_{T_i}}$} (ll)
	(ur) edge [right] node {$\,\, \Big( \t{iterated $F_{\tn}$'s and $\t{id}$'s} \Big) =: F_{\btn_{|T_i|}}$} (lr)
	(ul) edge [above] node {$\ups^{\s{D}}_{T_i}$} (ur)
	(ll) edge [below] node {$F\ups^{\s{C}}_{T_i}$} (lr)
	;
	\end{tikzpicture}
	\]
commutes for every $i \in I$. More precisely, the right hand map $F_{\btn_{|T_i|}}$ is given by $F_{\btn_0} := F_e$ and $F_{\btn_1} := \t{id}_F$ for $n=0,1$, and $F_{\btn_{n+1}} := F_\tn \circ (F_{\btn_n} \tn^{\s{D}} \id)$ for $n > 1$. We say that a lax $\c{T}$-normed morphism is \emph{strong} (resp. \emph{strict}) if the natural transformations $F_e$, $F_\tn$, and $F_{\btn_{T_i}}$ are all isomorphisms (resp. identities).
\end{defn}

\begin{defn}\label{defn:Ntransfm} Suppose that $\s{C}$ and $\s{D}$ are $\c{T}$-normed symmetric monoidal, and that $(F,F_\bullet) , \, (F',F'_\bullet) : \s{C} \rightrightarrows \s{D}$ is a pair of lax $\c{T}$-normed functors between them. An \emph{$\c{T}$-normed monoidal transformation $\omega : (F,F_\bullet) \Rightarrow (F',F'_\bullet)$} is a $G$-natural transformation $\omega : F \Rightarrow F'$ such that the usual monoidal transformation squares relating $F_e$, $F'_e$, $F_\tn$, and $F'_\tn$ to $\omega$ commute (cf. \cite[Ch. XI.2]{CWM}), and the square
	\[
	\begin{tikzpicture}
		\node(ul) {$\btn^{\s{D}}_{T_i}(FC_1 , \dots , FC_{|T_i|})$};
		\node(ur) at (7,0) {$\btn^{\s{D}}_{T_i} (F'C_1 , \dots , F'C_{|T_i|})$};
		\node(ll) at (0,-2) {$F\Big( \btn^{\s{C}}_{T_i} (C_1 , \dots , C_{|T_i|}) \Big)$};
		\node(lr) at (7,-2) {$F'\Big( \btn^{\s{C}}_{T_i} (C_1, \dots , C_{|T_i|}) \Big)$};
		\path[->]
		(ul) edge [left] node {$F_{\btn_{T_i}}$} (ll)
		(ur) edge [right] node {$F'_{\btn_{T_i}}$} (lr)
		(ul) edge [above] node {$\btn^{\s{D}}_{T} (\omega, \dots , \omega)$} (ur)
		(ll) edge [below] node {$\omega$} (lr)
		;
	\end{tikzpicture}
	\]
commutes for every $i \in I$.
\end{defn}

The $2$-category structure on $\ub{Cat}^G$ lifts to normed symmetric monoidal categories. The composite of lax maps $(H,H_\bullet) \circ (F,F_\bullet)$ is obtained by composing underlying functors and comparison data, e.g.
	\[
	(H \circ F)_{\otimes} = HF_\tn \circ H_\tn : HFC \tn HFC' \to H(FC \tn FC') \to HF(C \tn C').
	\]
Vertical and horizontal composites of transformations are computed in $\ub{Cat}^G$, and identities are also inherited from $\ub{Cat}^G$. 

\begin{nota}Let
	\[
	\c{T}\b{SMLax} = \t{the 2-category of $\c{T}$-NSMCs and lax monoidal functors} .
	\]
There are sub-$2$-categories 
	\[
	\c{T}\b{SMSt} \subset \c{T}\b{SMStg} \subset \c{T}\b{SMLax}
	\]
of strong and strict maps, and there is a forgetful $2$-functor $\c{T}\b{SMLax} \to \ub{Cat}^G$.
\end{nota}

\section{Functor NSMCs}\label{sec:FunNSMC}

With Definition \ref{defn:NSMC} in tow, we now consider an important class of examples. Suppose that $\s{C}$ is a symmetric monoidal object in $\b{Cat}^G$ (cf. Definition \ref{defn:symmmonobj}) and that $J$ is a small, right $G$-category. Consider the functor category $\b{Fun}(J,\s{C})$ of all $J$-diagrams in $\s{C}$. We shall explain how to make $\b{Fun}(J,\s{C})$ into a NSMC, and then analyze the special case where $J$ is the translation category of the group $G$.

\subsection{The general case}\label{subsec:FunNSMCgen} Keep $J$ and $\s{C}$ as above. To start, the category $\b{Fun}(J,\s{C})$ inherits a left $G$-action from $\s{C}$ and $J$ by composition:
	\[
	(g \cdot F)(-) = gF((-)g)	\quad\t{and}\quad	(g \cdot \eta)_{(-)} = g\eta_{(-)g} .
	\]
The category $\b{Fun}(J,\s{C})$ also inherits a levelwise symmetric monoidal structure from $\s{C}$. Given functors $F,H : J \rightrightarrows \s{C}$, let
	\[
	(F \otimes H)(j) = Fj \otimes Hj	\quad\t{and}\quad	(F \otimes H)(f : j \to j') = Ff \otimes Hf ,
	\]
and given natural transformations $\eta : F \Rightarrow F'$ and $\vartheta : H \Rightarrow H'$, let
	\[
	(\eta \otimes \vartheta)_j = \eta_j \otimes \vartheta_j : Fj \otimes Hj \to F' j \otimes H' j .
	\]
This defines a functor $\otimes : \b{Fun}(J,\s{C})^{\times 2} \to \b{Fun}(J,\s{C})$.

Let $e : J \to \s{C}$ be the constant functor valued at $e$.

For any $F,H,K : J \to \s{C}$, let $\alpha_{F,H,K} : (F \otimes H) \otimes K \to F \otimes (H \otimes K)$ be the natural isomorphism with $j$-component
	\[
	(\alpha_{F,H,K})_j = \alpha_{Fj,Hj,Kj} : (Fj \otimes Hj) \otimes Kj \to Fj \otimes (Hj \otimes Kj).
	\]
Then $\alpha$ is also natural in $F$, $H$, and $K$. Similarly, define natural isomorphisms $\lambda_F : e \otimes F \to F$, $\rho_F : F \otimes e \to F$, and $\beta_{F,H} : F \otimes H \to H \otimes F$ by
	\[
	(\lambda_F)_j = \lambda_{Fj}	\quad\t{and}\quad	(\rho_F)_j = \rho_{Fj}	\quad\t{and}\quad	(\beta_{F,H})_j = \beta_{Fj,Hj} .
	\]
These, too, are natural in $F$ and $H$.

Straightforward checks reveal that $\otimes$, $\alpha$, $\lambda$, $\rho$, and $\beta$ are $G$-equivariant, and that the coherence diagrams for a symmetric monoidal structure all commute. Therefore $\b{Fun}(J,\s{C})$ is a symmetric monoidal object in $\b{Cat}^G$. Note also that for every object $j \in J$, the (nonequivariant) evaluation functor $\t{ev}_j : \b{Fun}(J,\s{C}) \to \s{C}$ is strict symmetric monoidal.

We now consider external norms. As one might expect, the external norms of $\b{Fun}(J,\s{C})$ depend on the index category $J$. In general, the symmetric monoidal structure on $\b{Fun}(J,\s{C})$ extends to include an external $T$-norm $\btn_T$ and the corresponding untwistor $\ups_T$ if and only if
	\[
	\bigcup_{j \in J} \t{Stab}_H(j) \subset  \bigcap_{t \in T} \t{Stab}_H(t) = \Big\{ h \in H \, \Big| \, h \cdot (-) = \t{id} : T \to T \Big\},
	\]
i.e. if every $h \in H$ that stabilizes a single object of $J$ fixes all of $T$. We shall momentarily construct norms for all such $T$ (Construction \ref{const:TGCnorms}), and we shall also show what goes wrong when $T$ does not have this property (Nonexample \ref{nonex:FunJCnorm}).

\begin{const}\label{const:TGCnorms} Suppose that $H \subset G$ is a subgroup, and that $T$ is an ordered, finite $H$-set. Let the subgroup 
	\[
	\Gamma(T) = \{(h,\sigma(h)) \, | \, h \in H\} \subset G \times \Sigma_{|T|}
	\]
be the graph of the corresponding permutation representation $\sigma : H \to \Sigma_{\abs{T}}$ on $\{1, \dots , |T|\}$, so that $h \cdot i = \sigma(h)(i)$ for all $i \in \{1 , \dots , |T|\}$.

Now assume that $\bigcup_{j \in J} \t{Stab}_H(j) \subset \bigcap_{t \in T} \t{Stab}_H(t) = \t{ker}(\sigma)$, and choose a set of $H$-orbit representatives $\{j_a \, | \, a \in A\}$ for $\t{Ob}(J)$, considered as a right $H$-set.

\begin{enumerate}
		\item{}The norm map $\btn_{T} : \b{Fun}(J,\s{C})^{\times T} \to \b{Fun}(J,\s{C})$ is defined as follows.
			\begin{enumerate} 
				\item{}For an object $(C^1_\bullet , \dots , C^{|T|}_\bullet) \in \b{Fun}(J,\s{C})^{\times T}$ and $j = j_a h \in J$,
			\[
			\Big[\sbtn_T (C^1 , \dots , C^{|T|})\Big]_j := \sbtn_{|T|}(C^{\sigma(h)^{-1}1}_j , \dots , C^{\sigma(h)^{-1}|T|}_j),
			\]
		and for $f : j \to j'$, where $j = j_a h$ and $j' = j_b h'$,
			\[
			\Big[\sbtn_{T}(C^1 , \dots , C^{|T|})\Big]_{f} := \sigma(h') \circ \sigma(h)^{-1} \circ \sbtn_{|T|}(C^{\sigma(h)^{-1}1}_{f} , \dots , C^{\sigma(h)^{-1}|T|}_{f}).
			\]
			Here, the leading $\sigma(-)$'s are the symmetric monoidal coherence maps for $\s{C}$ that permute the factors of the tensor product by $\sigma(-)$. Note that the permutation $\sigma(h)$ is independent of the expression $j = j_a h$ because $\bigcup_{j \in J} \t{Stab}_H(j) \subset \t{ker}(\sigma)$.
			
				\item{}For a morphism $(f^1_\bullet , \dots f^{|T|}_\bullet) : (C^1_\bullet , \dots , C^{|T|}_\bullet ) \to (D^1_\bullet , \dots , D^{|T|}_\bullet)$ and an object $j = j_a h \in J$,
				\[
				\Big[\sbtn_{T}(f^1 , \dots , f^{|T|})\Big]_j := \sbtn_{|T|}( f^{\sigma(h)^{-1}1}_j , \dots , f^{\sigma(h)^{-1}|T|}_j).
				\]
			\end{enumerate}
		\item{}The untwistor $\upsilon = \upsilon_{T} : \btn_{T} \Rightarrow \btn_{|T|}$ is defined by
			\[
			\!\!\!\! (\upsilon_{C^1 , \dots , C^{|T|}})_j := \sigma(h)^{-1} : \sbtn_{|T|}(C^{\sigma(h)^{-1}1}_j , \dots , C^{\sigma(h)^{-1}|T|}_j) \to \sbtn_{|T|}(C^{1}_j , \dots , C^{|T|}_j),
			\]
		where $j = j_a h$, and $\sigma(h)^{-1}$ denotes the symmetric monoidal coherence map for $\s{C}$ that permutes the factors of $\btn_{|T|}$ by $\sigma(h)^{-1}$.
	\end{enumerate}

Unwinding the definitions shows that $\bigotimes_T$ is an external $T$-norm on $\b{Fun}(J,\s{C})$, and that $\ups_T$ is an untwistor for it. The required conditions all boil down to naturality and the coherence theorem for symmetric monoidal categories.\end{const}

On the other hand, if the inclusion $\bigcup_{j \in J} \t{Stab}_H(j) \subset \bigcap_{t \in T} \t{Stab}_H(t)$ does not hold, then one can produce examples of categories $\b{Fun}(J,\s{C})$ that do not support an external $T$-norm. We give one below.

\begin{nonex}\label{nonex:FunJCnorm} Let $J$ be a right $G$-category, let $H \subset G$ be a subgroup of $G$, and let $T$ be a finite $H$-action. Suppose that there is some $h_0 \in H$ and $j_0 \in J$ such that $j_0 \cdot h_0 = j_0$, but $h_0 \cdot (-) : T \to T$ is not the identity map. We shall show that the levelwise symmetric monoidal structure on $\b{Fun}(J,\s{C})$ does not generally extend to include a compatible $T$-norm and untwistor.

Let $\s{C} = (\b{Set} , \sqcup , \varnothing)$, and give $\s{C}$ a trivial $G$-action. Suppose for contradiction that we had a $T$-norm $\bigsqcup_T : \b{Fun}(J,\b{Set})^{\times T} \to \b{Fun}(J,\b{Set})$ and untwistor $\ups = \ups_T : \bigsqcup_T \Rightarrow \bigsqcup_{|T|}$. We consider $T$-fold coproducts of the terminal object $* : J \to \b{Set}$. The twisted equivariance diagram for $h_0$ is
	\[
	\begin{tikzpicture}
		\node(A) at (0,0) {$h_0 \cdot \bsq_T(*, \cdots , *)$};
		\node(B) at (4,0) {$\bsq_T(*, \cdots , *)$};
		\node(C) at (0,-4) {$h_0 \cdot \bsq_{|T|} (*,\cdots ,*)$};
		\node(D) at (4,-2) {$\bsq_{|T|}(*, \cdots , *)$};
		\node(E) at (4,-4) {$\bsq_{|T|}(*, \cdots , *)$};
		\path[->]
		(A) edge [above] node {$\t{id}$} (B)
		(A) edge [left] node {$h_0 \cdot \ups_{*,\cdots,*}$} (C)
		(B) edge [right] node {$\ups_{*,\cdots,*}$} (D)
		(D) edge [right] node {$\sigma(h_0)^{-1}_{*,\cdots,*}$} (E)
		(C) edge [below] node {$\t{id}$} (E)
		;
	\end{tikzpicture}
	\]
and evaluating at $j_0$ yields the equation
	\[
	(\ups_{*,\cdots,*})_{j_0} = (\ups_{*,\cdots,*})_{j_0 \cdot h_0} = (h_0 \cdot  \ups_{*,\cdots,*})_{j_0} = (\sigma(h_0)^{-1}_{*,\cdots,*})_{j_0} \circ (\ups_{*,\cdots,*})_{j_0} .
	\]
Since $\ups$ is an isomorphism, we must have that $(\sigma(h_0)^{-1}_{*,\cdots,*})_{j_0} = \t{id}$, but this is false because $(\sigma(h_0)^{-1}_{*,\cdots,*})_{j_0}$ is isomorphic to the permutation $h_0^{-1} \cdot (-) : T \to T$.
\end{nonex}

Thus, we conclude that $\b{Fun}(J,\s{C})$ supports external $T$-norms for those $T$ such that $\bigcup_{j \in J} \t{Stab}_H(j) \subset \bigcap_{t \in T} \t{Stab}_H(t)$, but no more in general.

\subsection{A special case}\label{ch:NSMC,sec:FunTGC}

Let $G$ be a finite group. The {translation category} of $G$ is the category whose objects are the elements of $G$, and which has a unique morphism $x \to y$ for any $x,y \in G$. We write $\bb{T} G$ for this category, and note that right multiplication in $G$ induces a right $G$-action on $\bb{T}G$. 

In this section, we study the functor category $\b{Fun}(\bb{T}G,\s{C})$, where $\s{C}$ is an ordinary symmetric monoidal category equipped with a trivial $G$-action. Such functor categories are  prototypical input to the infinite loop space machinery of Guillou--May--Merling--Osorno \cite{GM}, \cite{GMM}, \cite{GMMO}, and in what follows, we shall describe a surprising connection between their theory and the $G$-symmetric monoidal categories introduced by Hill and Hopkins \cite{HH}. Inspecting the definitions reveals that the norm from the $K$-fixed points of $\b{Fun}(\bb{T}G,\s{C})$ to its $H$-fixed points can be identified with the Hill-Hopkins-Ravenel norm $N_K^H : \b{Fun}(BK,\s{C}) \to \b{Fun}(BH,\s{C})$ (cf. Proposition \ref{prop:TGCnorms}), but we see no reason why this should be true from first principles. This discussion may be regarded as an elaboration on \cite[\S 3.2]{HH}.

\subsubsection{The NSMC $\b{Fun}(\bb{T}G,\s{C})$} We begin by establishing notation.

\begin{defn}\label{defn:transcat} The \emph{translation category} of a left $G$-set $X$ is the groupoid $\bb{T}X$ whose object set is $X$, and whose hom sets are $\bb{T}X(x,y) = \{g \in G \, | \, gx = y\}$. Composition is by group multiplication, and the unit $e \in G$ gives the identities. There is a functor $\bb{T} : \b{Set}^G \to \b{Cat}$ that sends a $G$-set $X$ to $\bb{T}X$, and sends a $G$-map $f : X \to Y$ to the functor $\bb{T}f : \bb{T}X \to \bb{T}Y$ defined by the formula $\bb{T}f(x) = f(x)$ on objects and $\bb{T}f(g : x \to y) = g : f(x) \to f(y)$ on morphisms. We shall sometimes write $\bb{T} = \bb{T}_G$ to emphasize that we are taking the translation category of a $G$-set.
\end{defn}

\begin{ex}The group $G$ acts on itself by left and right multiplication, and these actions interchange. Thus, we may regard $G$ asymmetrically as a left $G$-set equipped with a right $G$-action. Applying the functor $\bb{T}$ makes $\bb{T}G$ into a right $G$-category, with action maps
	\[
	(-) \cdot g := \bb{T}( (-) \cdot g) : \bb{T}G \to \bb{T}G.
	\]

Since $G$ is a transitive, free left $G$-set, it follows that for every $x,y \in \bb{T}G$, there is a unique morphism $! = yx^{-1} : x \to y$, and it is an isomorphism. For any $g \in G$, the functor $(-) \cdot g : \bb{T}G \to \bb{T}G$ sends the object $x \in \bb{T}G$ to $xg$ and the morphism $yx^{-1} : x \to y$ to $yx^{-1} : xg \to yg$.
\end{ex}

Now suppose that $\s{C}$ is an ordinary symmetric monoidal category, and regard it as a symmetric monoidal object in $\b{Cat}^G$ with a trivial $G$-action. Taking $J = \bb{T}G$ and specializing the constructions in \S\ref{subsec:FunNSMCgen} yields a NSMC $\b{Fun}(\bb{T}G,\s{C})$. Since $G$ acts trivially on $\s{C}$, the $G$-action on $\b{Fun}(\bb{T}G,\s{C})$ is given by
	\[
	g \cdot C_\bullet = C_\bullet \circ \bb{T}((-) \cdot g)	\quad\t{and}\quad	g \cdot \eta_\bullet = \eta_\bullet \circ \bb{T}((-) \cdot g)
	\]
on objects $C_\bullet$ and morphisms $\eta_\bullet$. Therefore $(g \cdot C)_\bullet = C_{\bullet g}$ and $(g  \cdot  \eta)_\bullet = \eta_{\bullet g}$. As before, the ordinary symmetric monoidal structure on $\b{Fun}(\bb{T}G,\s{C})$ is levelwise, and for any subgroup $H \subset G$ and finite, ordered $H$-set $T$, we can construct an external $T$-norm $\bigotimes_T : \b{Fun}(\bb{T}G,\s{C})^{\times T} \to \b{Fun}(\bb{T}G,\s{C})$ using the recipe in Construction \ref{const:TGCnorms}. Indeed, the $G$-action on $G = \t{Ob}(\bb{T}G)$ is free, so the only element $g \in G$ that fixes an object of $\bb{T}G$ is the identity. Thus $\b{Fun}(\bb{T}G,\s{C})$ is a NSMC with all norms.

\begin{rem} More generally, if $\s{C}$ is a symmetric monoidal object in $\b{Cat}^G$ (cf. Definition \ref{defn:symmmonobj}), then the $G$-action on the objects of $\b{Fun}(\bb{T}G,\s{C})$ becomes $g \cdot C_\bullet = g \cdot (-) \circ C_\bullet \circ \bb{T}((-) \cdot g)$, and similarly for morphisms. We have a normed symmetric monoidal structure exactly as above, but our analysis in \S\S\ref{subsubsec:FunTGCfp}--\ref{subsubsec:normFunTGC} below will break down.
\end{rem}

\subsubsection{The fixed points of $\b{Fun}(\bb{T}G,\s{C})$}\label{subsubsec:FunTGCfp} Next, we identify the $H$-fixed subcategory of $\b{Fun}(\bb{T}G,\s{C})$ with the category $\b{Fun}(BH,\s{C})$ of $H$-actions in $\s{C}$.

Let $H \subset G$ be a subgroup. Then the projection map $\pi : G \to G/H$ that sends $x \in G$ to $xH \in G/H$ determines a functor $\bb{T}\pi : \bb{T}G \to \bb{T}(G/H)$. Pulling back defines another functor
	\[
	\bb{T}\pi^* : \b{Fun}(\bb{T}(G/H) , \s{C}) \to \b{Fun}( \bb{T}G, \s{C}).
	\]
Since $\pi \circ (-)h = \pi$ for every $h \in H$, the functor $\bb{T}\pi^*$ lands in $\b{Fun}(\bb{T}G,\s{C})^H$. On the other hand, if the diagram $C_\bullet : \bb{T}G \to \s{C}$ is $H$-fixed, then it factors uniquely through $\bb{T}\pi$. For suppose $C_\bullet : \bb{T}G \to \s{C}$ is $H$-fixed. Then the factored functor $q{C}_\bullet : \bb{T}(G/H) \to \s{C}$ is defined as follows:
	\begin{enumerate}
		\item{}For any object $aH \in \bb{T}(G/H)$, choose a representative $ah \in aH$ and set
			\[
			q{C}_{aH} := C_{ah} \, (= C_a ).
			\]
		\item{}For any morphism $g : aH \to gaH$ in $\bb{T}(G/H)$, choose a representative $ah \in aH$ and set
			\[
			q{C}_{g : aH \to gaH} := C_{g : ah \to gah} \, (= C_{g : a \to ga} ).
			\]
	\end{enumerate}
Similarly, if $\eta : C_\bullet \Rightarrow D_\bullet$ is a $H$-fixed natural transformation, then $q{\eta}_{aH} := \eta_a$ defines a natural transformation $q{\eta} : q{C}_\bullet \Rightarrow q{D}_\bullet$. It follows that there are strictly inverse functors
	\[
	\bb{T}\pi^* : \b{Fun}(\bb{T}(G/H),\s{C}) \stackrel{\cong}{\rightleftarrows} \b{Fun}(\bb{T}G,\s{C})^H : q.
	\]
	
To proceed further in our identification of $\b{Fun}(\bb{T}G,\s{C})^H$, we recall an observation of Hill--Hopkins--Ravenel \cite{HHR}. Let the functor 
	\[
	s : \bb{T}_H(H/H) \to \bb{T}_G(G/H)
	\]
be the inclusion of $\bb{T}_H(H/H)$ as the automorphisms of the coset $eH \in \bb{T}_G(G/H)$. The functor $s$ is an equivalence of categories, and we can construct an explicit deformation retraction
	\[
	\bb{T}_H(H/H) \leftarrow \bb{T}_G(G/H) : r
	\]
by choosing a set of $G/H$ coset representatives $\{e = g_1 , \dots , g_{|G:H|}\}$, and then setting $r(g_i H) = H$ on objects and $r(g : g_i H \to g_j H) = g_j^{-1} g g_i : H \to H$ on morphisms. Then $s \circ r \simeq \t{id}$, and $r \circ s = \t{id}$ because $g_1 = e$. The equivalence $\bb{T}_H(H/H) \simeq \bb{T}_G(G/H)$ induces an equivalence
	\[
	r^* : \b{Fun}(\bb{T}_H(H/H) , \s{C}) \stackrel{\simeq}{\rightleftarrows} \b{Fun}(\bb{T}_G(G/H) , \s{C}) : s^*,
	\]
and the category $\b{Fun}(\bb{T}_H(H/H),\s{C})$ is isomorphic to the category $\b{Fun}(BH,\s{C})$ of $H$-actions in $\s{C}$. In summary, we find that $\b{Fun}(BH,\s{C}) \simeq \b{Fun}(\bb{T}G,\s{C})^H$. 

\subsubsection{The norms of $\b{Fun}(\bb{T}G,\s{C})$}\label{subsubsec:normFunTGC} We now define norms between the fixed point subcategories of $\b{Fun}(\bb{T}G,\s{C})$. Suppose more generally that $\s{D}$ is a $G$-category and that $K \subset H \subset G$ are subgroups. Choose a set of coset representatives $e = h_1 , \dots, h_{\abs{H:K}}$ for $H/K$. This orders $H/K$, and the corresponding permutation representation on $\{ 1 , \dots , \abs{H:K} \}$ is defined by
	\[
	h \cdot h_i K = h_{\sigma(h)(i)} K.
	\]
We can form the $H$-category $\s{D}^{\times H/K}$ as in Definition \ref{def:extnorm}. Its action is given on objects by the formula $h \cdot (d_1 , \dots , d_{\abs{H:K}}) = (h \cdot d_{\sigma(h)^{-1}1} , \dots , h \cdot d_{\sigma(h)^{-1}\abs{H:K}})$, and similarly for morphisms. Crucially, there is a pair of inverse functors
	\[
	\Delta^{\t{tw}} : \s{D}^K \rightleftarrows (\s{D}^{\times H/K})^H : \t{ev}_K ,
	\]
where $\t{ev}_K$ is the first coordinate projection, and $\Delta^{\t{tw}}$ is a twisted diagonal map. The functor $\Delta^{\t{tw}}$ sends $d \in \s{D}$ to $(d , h_2 d, \dots, h_{\abs{H:K}} d)$ and similarly for morphisms.

Combining twisted diagonals with external norm maps allows one to define norms in the usual sense.

\begin{defn}\label{defn:internalnorm} Suppose $\s{D}$ is a $G$-category equipped with an external $H/K$ norm $\btn_{H/K} : \s{D}^{\times H/K} \to \s{D}$, and let $\s{D}^{\times H/K}$ and $\Delta^{\t{tw}}$ be defined as above. The \emph{norm} $\t{n}_K^H : \s{D}^K \to \s{D}^H$ is the composite
	\[
	\t{n}_K^H = \sbtn_{H/K} \circ \Delta^{\t{tw}} : \s{D}^K \stackrel{\cong}{\to} (\s{D}^{\times H/K})^H \to \s{D}^H.
	\]
\end{defn}

If $\s{D} = \b{Fun}(\bb{T}G,\s{C})$, then the norm $\t{n}_K^H$ can be identified with the Hill-Hopkins-Ravenel norm $N_K^H : \b{Fun}(BK,\s{C}) \to \b{Fun}(BH,\s{C})$. More precisely, we have the following result.

\begin{prop}\label{prop:TGCnorms} Suppose that $G$ is a finite group, $K \subsetneq H \subset G$ are subgroups, and that $\s{C}$ is a nonequivariant symmetric monoidal category. Then there is a commutative diagram
	\[
	\begin{tikzpicture}[scale=0.8]
		\node(KC) {$\b{Fun}(BK,\s{C})$};
		\node(HC) at (11,0) {$\b{Fun}(BH,\s{C})$};
		\node(TGK) at (0,2) {$\b{Fun}(\bb{T}(G/K),\s{C})$};
		\node(TGH) at (11,2) {$\b{Fun}(\bb{T}(G/H),\s{C})$};
		\node(TGCK) at (0,4) {$\b{Fun}(\bb{T}G,\s{C})^K$};
		\node(nTGC) at (5.5,4) {$\Big( \b{Fun}(\bb{T}G,\s{C})^{\times H/K} \Big)^H$};
		\node(TGCH) at (11,4) {$\b{Fun}(\bb{T}G,\s{C})^H$};
		\path[->]
		(KC) edge [below] node {$N_K^H$} (HC)
		(KC) edge [left] node {$r^*$} (TGK)
		(TGH) edge [right] node {$s^*$} (HC)
		(TGK) edge [below] node {$p^{\tn}_*$} (TGH)
		(TGK.120) edge [left] node {$\bb{T}\pi^*$} (TGCK.-119)
		(TGCK.-61) edge [right] node {$q$} (TGK.60)
		(TGH.120) edge [left] node {$\bb{T}\pi^*$} (TGCH.-119)
		(TGCH.-61) edge [right] node {$q$} (TGH.60)
		(TGCK) edge [above] node {$\Delta^{\t{tw}}$} (nTGC)
		(nTGC) edge [above] node {$\btn_{H/K}$} (TGCH)
		(TGCK) edge [bend left, above] node {$\t{n}_K^H$} (TGCH)
		;
	\end{tikzpicture}
	\]
where $p : \bb{T}(G/K) \to \bb{T}(G/H)$ is induced by the quotient $G/K \to G/H$, $p_*^\otimes$ is its monoidal pushforward, and $N_K^H$ is the norm functor.

Consequently, the fixed points of $\b{Fun}(\bb{T}G,\s{C})$ are equivalent to the symmetric monoidal Mackey functor $\underline{\s{C}}^{(-)}$, whose value at $G/H$ is the category of $H$-actions in $\s{C}$, and whose transfers are Hill--Hopkins--Ravenel norms (cf. \cite[Example 2.6]{HH}).
\end{prop}

The proof is straightforward once the correct definitions have been made, but it requires a little care because some of these functors above are defined using noncanonical choices. We review some definitions and outline one verification below.

\subsubsection{Monoidal pushforwards, Hill--Hopkins--Ravenel norms, and the proof of Proposition \ref{prop:TGCnorms}} \label{subsubsec:monoidalpsh}

Suppose that $\s{C}$ is a nonequivariant symmetric monoidal category, and recall the following definition \cite[Definition A.24]{HHR}.

\begin{defn}\label{def:fincovcat} A \emph{finite covering category} is a functor $p : I \to J$ such that
	\begin{enumerate}
		\item{}For every morphism $f : j \to j'$ in $J$ and object $i \in p^{-1}(j)$, there is a unique $I$-morphism $\til{f}$ such that $\t{dom} \til{f} =  i$ and $p \til{f} = f$.
		\item{}For every morphism $f : j \to j'$ in $J$ and object $i' \in p^{-1}(j')$, there is a unique $I$-morphism $\til{f}$ such that $\t{cod} \til{f} = i'$ and $p \til{f} = f$.
		\item{}For every object $j \in J$, the fiber $p^{-1}(j) \subset \t{Ob}(I)$ is finite.
	\end{enumerate}
\end{defn}

In \cite{Bohmann}, \cite{HHR}, the authors explain how to construct a monoidal pushforward $p^\tn_* : \b{Fun}(I,\s{C}) \to \b{Fun}(J,\s{C})$ associated to any finite covering category $p : I \to J$. For their purposes, it was not necessary to track the orderings in tensor products too carefully, and therefore those details were rightly suppressed. Unfortunately, our present work demands attention to these matters, because twisted equivariance for $\ups_T$ is entirely about the relationship between a group action and symmetric monoidal permutation maps. Thus, we review the construction of $p^\tn_*$, with a focus on orderings.

\begin{conv}We shall assume that every finite covering category $p : I \to J$ comes equipped with a chosen linear ordering on every fiber $p^{-1}(j)$.
\end{conv}

Suppose $p : I \to J$ is a finite covering category and keep notation as in Definition \ref{def:fincovcat}. We write $\til{f}_i$ for the unique lift of $f$ starting at $i$, and we define $f \cdot i := \t{cod} \til{f}_i$. For any $f : j \to j'$ in $J$, we obtain a set bijection $f \cdot (-) : p^{-1}(j) \to p^{-1}(j')$, but this map need not respect the orderings of $p^{-1}(j)$ and $p^{-1}(j')$. So suppose that the fibers $p^{-1}(j)$ and $p^{-1}(j')$ both have cardinality $n$, and write
	\[
	p^{-1}(j) = \{ i_1 < \dots < i_n \}	\quad\text{and}\quad	p^{-1}(j') = \{ i'_1 < \dots < i'_n \}.
	\]
Then we define the permutation $\sigma(f) \in \Sigma_n$ by the formula $f \cdot i_k = i'_{\sigma(f)k}$. Equivalently, there is a lift $\til{f}_{i_k} : i_k \to i'_{\sigma(f)k}$ of $f$.

Recall (Definition \ref{defn:symmmonobj}) that the \emph{standard tensor products} on $\s{C}$ are $\btn_0 := e$, $\btn_1 := \t{id}$, $\btn_2 := \tn$, and $\btn_{n+1} := \tn \circ (\btn_n \times \t{id})$ for $n \geq 2$.

\begin{defn}Keep notation as above. For any symmetric monoidal category $\s{C}$ and finite covering category $p : I \to J$, the \emph{monoidal pushforward} functor $p^\tn_* : \b{Fun}(I,\s{C}) \to \b{Fun}(J , \s{C})$ is defined as follows.
	\begin{enumerate}[label=(\alph*)]
		\item{}Given a functor $X : I \to \s{C}$, we define $p^\tn_*X : J \to \s{C}$ on objects $j \in J$ by
			\[
			(p^\tn_*X)(j) := \begin{array}{c}\btn_n\end{array}\!\!\!{} \Big( X(i_1) , \dots , X(i_n) \Big),
			\]
		where $p^{-1}(j) = \{ i_1 < \dots < i_n \}$ and $\btn_n$ is the standard $n$-fold tensor product. Then, given $f: j \to j'$ in $J$, we define $p^\tn_*(f) : p^\tn_*(j) \to p^\tn_*(j')$ to be the composite
			\[
			\begin{tikzpicture}
				\node(A) at (0,0) {$\btn_n \Big( X(i_k) \Big)$};
				\node(B) at (5.5,0) {$\btn_n \Big( X(i'_{\sigma(f)k}) \Big)$};
				\node(C) at (10,0) {$\btn_n \Big( X(i'_k) \Big)$};
				\path[->]
				(A) edge [above] node {$\btn_n \Big( X(\til{f}_{i_k}) \Big)$} (B)
				(B) edge [above] node {$\sigma(f)$} (C)
				;
			\end{tikzpicture}
			\]
		where $\sigma(f)$ is the symmetric monoidal coherence map for $\s{C}$ that permutes the factors of the tensor product by $\sigma(f)$.
		\item{}Given a natural transformation $\eta_\bullet : X_\bullet \Rightarrow Y_\bullet$, we define
			\[
			(p^\tn_*\eta)_j := \begin{array}{c}\btn_n\end{array}\!\!\!{} \Big( \eta_{i_1} , \dots , \eta_{i_n} \Big),
			\]
		where $p^{-1}(j) = \{ i_1 < \dots < i_n \}$.
	\end{enumerate}
\end{defn}

The monoidal pushforward is a key ingredient in the definition of the norm functor $N_K^H$, which we now recall (cf. \cite[Definition A.52]{HHR}).

\begin{defn}For any subgroups $K \subset H \subset G$, the \emph{norm functor} $N_K^H : \b{Fun}(BK,\s{C}) \to \b{Fun}(BH,\s{C})$ is the composite 
	\[
	p^\tn_* \circ r^* : \b{Fun}(BK,\s{C}) \to \b{Fun}(\bb{T}_H(H/K),\s{C}) \to \b{Fun}(BH,\s{C}),
	\]
where $r : \bb{T}_H(H/K) \to \bb{T}_K(K/K) = BK$ is the retraction described in \S\ref{subsubsec:FunTGCfp}, and $p^\tn_*$ is the monoidal pushforward for $p : \bb{T}_H(H/K \to H/H)$.
\end{defn}

From here, the claims in Proposition \ref{prop:TGCnorms} follow by unravelling the definitions. We sketch one possible route, but leave the details to the interested reader.

\begin{proof}[Sketch of proof of Proposition \ref{prop:TGCnorms}] The point is to make compatible choices of coset representatives and orderings. 

First, choose sets $\{e = g_1 , \dots , g_{|G:H|}\}$ and $\{e = h_1 , \dots , h_{|H:K|} \}$ of $G/H$ and $H/K$ coset representatives, and give the orbits $G/H$ and $H/K$ the corresponding orders. We obtain a set $\{ g_i h_j \, | \, 1 \leq i \leq |G:H| \, , \, 1 \leq j \leq |H:K| \}$ of $G/K$ coset representatives, and we order $G/K$ lexicographically as follows:
	\[
	K < h_2 K < \dots < h_{|H:K|} K < g_2 K < g_2 h_2 K < \dots < g_2 h_{|H:K|} K < \cdots .
	\]
From here, we
	\begin{enumerate}[label=(\alph*)]
		\item{}use the relation $h \cdot h_i K = h_{\sigma(h) i} K$ to define $\Gamma(H/K) = \{(h,\sigma(h)) \, | \, h \in H\}$, and give $\b{Fun}(\bb{T}G,\s{C})^{\times H/K}$ the diagonal $H$-action twisted by $\sigma$,
		\item{}construct the norm map $\btn_{H/K} : \b{Fun}(\bb{T}G,\s{C})^{\times H/K} \to \b{Fun}(\bb{T}G,\s{C})$ as in Construction \ref{const:TGCnorms}, using the $G/H$ coset representatives $g_i$,
		\item{}define $p^\tn_* : \b{Fun}(\bb{T}(G/K),\s{C}) \to \b{Fun}(\bb{T}(G/H),\s{C})$ using the order on the fibers of $p : \bb{T}(G/K) \to \bb{T}(G/H)$ induced by the order on $G/K$,
		\item{}define $r : \bb{T}_G(G/K) \to \bb{T}_K(K/K)$ using the coset representatives $g_i h_j$, and
		\item{}for the norm $N_K^H : K\s{C} \to H\s{C}$, we define $r : \bb{T}_H(H/K) \to \bb{T}_K(K/K)$ using the coset representatives $h_j$, and we use the order on $H/K$ to construct the monoidal pushforward $p^\tn_* : \b{Fun}(\bb{T}_H(H/K),\s{C}) \to \b{Fun}(\bb{T}_H(H/H),\s{C})$.
	\end{enumerate}
With these choices, one can verify that the diagram in Proposition \ref{prop:TGCnorms} commutes. For the upper square, we find it easiest to show that $p^{\tn}_*$ and the composite $q \circ \btn_{H/K} \circ \Delta^{\t{tw}} \circ \bb{T}\pi^*$ are equal.
\end{proof}

\section{$2$-categories of operadic algebras}\label{sec:opalg2cats}

We now return to general theory, and introduce some more definitions needed in the operadic coherence theorem for NSMCs. Suppose $\s{O}$ is an operad in $\b{Cat}^G$, so that $\s{O}$ can act on $G$-categories. Then the collection of all $\s{O}$-algebras fits into a variety of $2$-categories, which are precisely analogous to the $2$-categories considered in \S\ref{ch:NSMC,sec:defn}. We spell out their structure below, starting with objects.

\begin{defn} Suppose $\s{C}$ is a $G$-category. The \emph{endomorphism operad} $\b{End}(\s{C})$ is the operad in $\b{Cat}^G$ whose $n$th level is the $G$-category $\b{Cat}_G(\s{C}^{\times n},\s{C})$. For any operad $\s{O}$ in $\b{Cat}^G$, a \emph{strict $\s{O}$-algebra} in $\b{Cat}^G$ is a $G$-category $\s{C}$ equipped with an operad map $\abs{\cdot}_\s{C} : \s{O} \to \b{End}(\s{C})$.
\end{defn}

We think of the map $\abs{\cdot}_{\s{C}}$ as realizing an abstract operation symbol $f \in \s{O}$ as an actual operation on $\s{C}$. In contrast to Definition \ref{defn:NSMC}, an operad action $\s{O} \to \b{End}(\s{C})$ parametrizes much more than just a small set of generating data, because $\s{O}$ is closed under composition.

Next, we consider $1$-morphisms.

\begin{defn}\label{defn:laxopmap}Suppose that $\s{O}$ is an operad in $\b{Cat}^G$, and that $\s{C}$ and $\s{D}$ are strict $\s{O}$-algebras in $\b{Cat}^G$. A \emph{lax $\s{O}$-algebra morphism} $(F,\partial_\bullet) : \s{C} \to \s{D}$ consists of:
	\begin{enumerate}
		\item{}a $G$-functor $F : \s{C} \to \s{D}$, and
		\item{}for each $n \geq 0$ and $x \in \s{O}(n)$, a natural transformation $(\vd_n)_x : \abs{x}_{\s{D}} \circ F^{\times n} \Rightarrow F \circ \abs{x}_{\s{C}}$ of functors $\s{C}^{\times n} \to \s{D}$.
	\end{enumerate}
These data are required to satisfy the following conditions:
	\begin{enumerate}[label=(\roman*)]
		\item{}for each $n \geq 0$, the maps $(\vd_n)_x$ vary naturally in $x \in \s{O}(n)$,
		\item{}for each $n \geq 0$, $x \in \s{O}(n)$, and $(g,\sigma) \in G \times \Sigma_n$, the equation $(g,\sigma) \cdot (\vd_n)_x = (\vd_n)_{(g,\sigma) \cdot x}$ holds, i.e. $(\vd_n)_x$ is $(G \times \Sigma_n)$-equivariant in $x$,
		\item{}$(\vd_1)_{\t{id}} = \t{id}_F : \abs{\t{id}}_{\s{D}} \circ F^{\times 1} \Rightarrow F \circ \abs{\t{id}}_{\s{C}}$, and
		\item{}for any $y \in \s{O}(m)$ and $x_i \in \s{O}(k_i)$, the transformation 
		\[
		(\vd_{k_1 + \cdots + k_m})_{\gamma(y;x_1 , \dots , x_m)} : \abs{\gamma(y;x_1, \dots , x_m)}_{\s{D}} \circ F^{\times k_1 + \cdots + k_m} \Rightarrow F \circ \abs{\gamma(y;x_1, \dots , x_m)}_{\s{C}}
		\]
		is equal to the composite
		\[
		\begin{tikzpicture}[scale=0.75]
			\node(A) {$\abs{y}_{\s{D}} \circ \Big( (\abs{x_1}_{\s{D}} \circ F^{\times k_1}) \times \cdots \times (\abs{x_m}_{\s{D}} \circ F^{\times k_m}) \Big)$};
			\node(B) at (0,-2) {$(\abs{y}_{\s{D}} \circ F^{\times m}) \circ \Big( \abs{x_1}_{\s{C}} \times \cdots \times \abs{x_m}_{\s{C}} \Big)$};
			\node(C) at (0,-4) {$F \circ \abs{y}_{\s{C}} \circ \Big( \abs{x_1}_{\s{C}} \times \cdots \times \abs{x_m}_{\s{C}} \Big).$};
			
			\node(1) at (2.55,-1) {$\Downarrow \, \t{id}_{\abs{y}_{\s{D}}} \circ \Big( (\vd_{k_1})_{x_1} \times \cdots \times (\vd_{k_m})_{x_m} \Big) $};
			\node(2) at (2.4,-3) {$\Downarrow \, (\vd_m)_y \circ \Big( \t{id}_{\abs{x_1}_{\s{C}}} \times \cdots \times \t{id}_{\abs{x_m}_{\s{C}}} \Big)$};
			;
		\end{tikzpicture}
		\]
	\end{enumerate}
A \emph{pseudomorphism} (resp. \emph{strict morphism}) is a lax morphism such that $(\vd_n)_x$ is an isomorphism (resp. identity) for every $n \geq 0$ and $x \in \s{O}(n)$.
\end{defn}

The natural transformations $\vd$ play the role of the comparison maps $F_e$, $F_\otimes$, and $F_{\bigotimes_T}$ in Definition \ref{defn:laxnormedfunctor}, but as before, we have specified far more than a small set of generating data. Condition (2) gives a comparison map for every operation, and conditions (i) -- (iv) ensure that these maps are suitably compatible.

Now we consider $2$-morphisms.

\begin{defn}\label{defn:Otransfm}Suppose that $\s{O}$ is an operad in $\b{Cat}^G$, that $\s{C}$ and $\s{D}$ are strict $\s{O}$-algebras, and that $(F,\vd_\bullet) , (F' , \vd'_\bullet) : \s{C} \rightrightarrows \s{D}$ is a pair of lax $\s{O}$-algebra morphisms. An \emph{$\s{O}$-transformation $\omega : (F,\vd_\bullet) \Rightarrow (F',\vd'_\bullet)$} is a $G$-natural transformation $\omega : F \Rightarrow F'$ such that for every $n \geq 0$ and $x \in \s{O}(n)$, the composite transformations
	\[
	\begin{tikzpicture}[scale=0.75]
		\node(11) at (-2,0) {$|x|_{\s{D}} \circ F^{\times n}$};
		\node(21) at (-2,-2) {$F \circ \abs{x}_{\s{C}}$};
		\node(31) at (-2,-4) {$F' \circ \abs{x}_{\s{C}}$};
		\node(12) at (4,0) {$\abs{x}_{\s{D}} \circ F^{\times n}$};
		\node(22) at (4,-2) {$\abs{x}_{\s{D}} \circ (F')^{\times n}$};
		\node(32) at (4,-4) {$F' \circ \abs{x}_{\s{C}}$};
		
		\node(eq) at (0.8,-2) {$\text{and}$};
		
		\node(c12) at (-2.55,-1) {$(\vd_n)_x \, \Downarrow$};
		\node(c23) at (-2.92,-3) {$\omega \circ \t{id}_{\abs{x}_{\s{C}}} \, \Downarrow$};
		\node(d12) at (5,-1) {$\Downarrow \, \t{id}_{\abs{x}_{\s{D}}} \circ \omega^{\times n}$};
		\node(d23) at (4.35,-3) {$\Downarrow \, (\vd'_n)_x$};
	\end{tikzpicture}
	\]
are equal.
\end{defn}

Again, we have specified more compatibility conditions than in Definition \ref{defn:Ntransfm}, because we range over all $n \geq 0$ and $x \in \s{O}(n)$.

As with NSMCs, the $2$-category structure on $\ub{Cat}^G$ lifts to $\s{O}$-algebras. The composite of lax $\s{O}$-morphisms $(H,\ve_\bullet) \circ (F,\vd_\bullet)$ is obtained by composing underlying functors and comparison data, e.g.
	\[
	[\t{id}_H \circ (\vd_n)_x] \vc [ (\ve_n)_x \circ \t{id}_{F^{\times n}}] : \abs{x} \circ H^{\times n} \circ F^{\times n} \Rightarrow H \circ \abs{x} \circ F^{\times n} \Rightarrow H \circ F \circ \abs{x},
	\]
and the vertical and horizontal composites of transformations are computed in $\ub{Cat}^G$. Identities of both sorts are also inherited from $\ub{Cat}^G$. 

\begin{nota} Let $\s{O}\t{-}\b{AlgLax}$ be the 2-category of all strict $\s{O}$-algebras in $\b{Cat}^G$, lax $\s{O}$-morphisms, and $\s{O}$-transformations between them. There are sub-$2$-categories
	\[
	\s{O}\t{-}\b{AlgSt} \subset \s{O}\t{-}\b{AlgPs} \subset \s{O}\t{-}\b{AlgLax} ,
	\]
whose $1$-morphisms are pseudo or strict, and there is also a forgetful $2$-functor $\s{O}\t{-}\b{AlgLax} \to \ub{Cat}^G$.
\end{nota}

We conclude this section with a brief digression. For future reference, it will be convenient to recast conditions (ii) -- (iv) in Definition \ref{defn:laxopmap} in more operadic terms. In a certain sense, they state that the assignment $x \mapsto (\vd_n)_x$ is an operad map, which will streamline the discussion  in Proposition \ref{prop:evbijfun}.

\begin{defn}\label{defn:laxmapop}Suppose that $\s{C}$ and $\s{D}$ are strict $\s{O}$-algebras and that $F : \s{C} \to \s{D}$ is a $G$-functor. We define an operad $\c{L}ax = \c{L}ax(\s{O},\s{C},\s{D},F)$ in $\b{Set}^G$ as follows.
	\begin{enumerate}
		\item{}For each integer $n \geq 0$, let $\c{L}ax(n)$ be the set of pairs $(x, \xi)$, where $x \in \s{O}(n)$ and $\xi : \abs{x}_{\s{D}} \circ F^{\times n} \Rightarrow F \circ \abs{x}_{\s{C}}$ is a natural transformation in $\b{Cat}_G(\s{C}^{\times n} , \s{D})$. The $G \times \Sigma_n$-action is $(g,\sigma) \cdot (x,\xi) = \big( (g,\sigma) \cdot x , (g,\sigma) \cdot \xi \big)$.
		\item{}Define the identity for $\c{L}ax$ to be the pair $(\t{id}, \t{id}_F)$.
		\item{}Define composition maps
			\[
			\gamma_{\c{L}} : \c{L}ax(k) \times \c{L}ax(j_1) \times \cdots \times \c{L}ax(j_k) \to \c{L}ax(j_1 + \cdots + j_k)
			\]
		by setting $\gamma_{\c{L}}((z,\zeta) ; (x_1,\xi_1) , \dots , (x_k,\xi_k))$ equal to 
		\[
		\Big( \gamma_{\s{O}}(z;x_1,\dots,x_k) \,\,,\,\,  [\zeta \circ  (\t{id}_{\abs{x_1}_{\s{C}}} \times \cdots \times \t{id}_{\abs{x_k}_{\s{C}}} )] \vc [ \t{id}_{\abs{z}_{\s{D}}} \circ ( \xi_1 \times \cdots \times \xi_k)] \Big),
		\]
	where $\circ$ denotes horizontal composition and $\bullet$ denotes vertical composition.
	\end{enumerate}
The first coordinate projection defines a map $\pi_1 : \c{L}ax \to \t{Ob}(\s{O})$ of $G$-operads.
\end{defn}

\begin{rem}There are suboperads $\t{Ob}\s{O} \cong \c{S}t \subset \c{P}s \subset \c{L}ax$ obtained by restricting all $\xi$'s to be identity transformations and natural isomorphisms, respectively.
\end{rem}

The following is a quick check of definitions.

\begin{prop}\label{prop:laxopmap} Suppose that $\s{C}$ and $\s{D}$ are strict $\s{O}$-algebras, that $F : \s{C} \to \s{D}$ is a $G$-functor, and that for each $n \geq 0$ and $x \in \s{O}(n)$, we are given a natural transformation $(\vd_n)_x : \abs{x}_{\s{D}} \circ F^{\times n} \Rightarrow F \circ \abs{x}_{\s{C}}$.\footnote{For any $n \geq 0$, $\vd_n = ((\vd_n)_x)_{x \in \s{O}(n)}$ is a tuple of natural transformations, but we do \emph{not} assume that $\vd_n$ is natural in $x$.} Then conditions (ii) -- (iv) of Definition \ref{defn:laxopmap} hold if and only if the section $s : \t{Ob}(\s{O}) \to \c{L}ax$ of $\pi_1 : \c{L}ax \to \t{Ob}(\s{O})$ defined by $s_n(x) := (x , (\vd_n)_x)$ is a map of $G$-operads.
\end{prop}

\begin{proof} The map $s$ is $G \times \Sigma$ equivariant if and only if condition (ii) holds, it preserves the unit if and only if condition (iii) holds, and it preserves composition if and only if condition (iv) holds.
\end{proof}

\section{The coherence theorem}\label{ch:NSMC,sec:cohthm}

The classical Mac Lane coherence theorem states that all sensible diagrams built up from the associativity, unitality, and symmetry isomorphisms in a symmetric monoidal category must commute. In this section and the next, we shall generalize and reinterpret this statement operadically. In particular, we will show that any $\c{T}$-normed symmetric monoidal structure on a $G$-category $\s{C}$ extends to an action by an operad $\c{SM}_{\c{T}}$, which parametrizes all diagrams that must commute for formal reasons. We make this statement precise in this section (Theorem \ref{thm:coh}), but we defer the proof to \S\ref{sec:pfcoh}.

\subsection{Permutative categories} To explain our formulation of the coherence theorem, we review a well-known example (cf. \cite{MayPerm}). Recall that a \emph{permutative category} is a strictly associative and unital symmetric monoidal category, i.e. a symmetric monoidal category for which the associativity and unit isomorphisms are identity transformations. Though permutative categories are rare in nature, every symmetric monoidal category can be rigidified to a permutative one, so there is no loss of generality in restricting attention to permutative categories.

If $\s{C}$ is a permutative category, then the only $n$-ary operations generated from its product $\otimes^{\s{C}}$ are the $n$-ary tensor products of the form
	\[
	x_{\sigma^{-1}1} \otimes^{\s{C}} x_{\sigma^{-1}2} \otimes^{\s{C}} \cdots \otimes^{\s{C}} x_{\sigma^{-1}n} ,
	\]
where $\sigma \in \Sigma_n$ is a permutation of $n$ letters. In this case, the coherence theorem states that for any $\sigma,\tau \in \Sigma_n$, there is a unique natural isomorphism
	\[
	x_{\sigma^{-1}1} \otimes^{\s{C}} x_{\sigma^{-1}2} \otimes^{\s{C}} \cdots \otimes^{\s{C}} x_{\sigma^{-1}n} \Rightarrow
	x_{\tau^{-1}1} \otimes^{\s{C}} x_{\tau^{-1}2} \otimes^{\s{C}} \cdots \otimes^{\s{C}} x_{\tau^{-1}n}
	\]
built up from instances of $\beta^{\s{C}}$.

We can interpret this situation in terms of operad actions. Recall that the associativity operad $\b{As}$ is the operad in $\b{Set}$ whose $n$th level is $\Sigma_n$. We think of the permutation $\sigma \in \Sigma_n$ as the $n$-fold product
	\[
	x_{\sigma^{-1}1} x_{\sigma^{-1}2} \cdots x_{\sigma^{-1}n} .
	\]
Thus, the operad $\b{As}$ suffices to encode the operations of a permutative category, but it cannot encode the coherence isomorphisms because it is discrete. To get these isomorphisms, we adjoin a unique isomorphism $\sigma \to \tau$ between any two permutations $\sigma, \tau \in \Sigma_n$ in the same level. We isolate this construction.

Let $\b{Cat}$ denote the category of small categories. There is a functor
	\[
	\t{Ob} : \b{Cat} \to \b{Set}
	\]
that sends a small category to its set of objects. In analogy to the underlying set functor $U : \b{Top} \to \b{Set}$, the functor $\t{Ob}$ has a left adjoint $(-)^{\t{disc}} : \b{Set} \to \b{Cat}$, which sends a set $X$ to the discrete category $X^{\t{disc}}$. The category $X^{\t{disc}}$ has object set $X$, and no nonidentity morphisms. Now we dualize.

\begin{defn}\label{defn:codiscrete} For any set $X$, the \emph{codiscrete category} $\til{X}$ is the small category with object set $X$, and a unique morphism $x \to y$ for each pair of elements $x,y \in X$. There is an adjunction
	\[
	\t{Ob} : \b{Cat} \rightleftarrows \b{Set} : \til{(-)} .
	\]
\end{defn}

Being a right adjoint, the functor $\til{(-)}$ preserves limits, and therefore it preserves operads. Furthermore, the functor $\til{(-)}$ sends every nonempty set to a contractible groupoid in which every diagram commutes.

The \emph{permutativity operad} $\s{P}$ is the operad obtained by applying $\til{(-)}$ levelwise to $\b{As}$. If a category $\s{C}$ is an algebra over $\s{P}$, i.e. there is a map $\abs{\cdot}_{\s{C}} : \s{P} \to \b{End}(\s{C})$, then evaluating at
	\[
	\otimes = \t{id}_2 \in \Sigma_2	\quad\t{and}\quad	e = \t{id}_0 \in \Sigma_0	\quad\t{and}\quad \beta : \t{id}_2 \to (12) \in \til{\Sigma_2}
	\]
gives a permutative category structure on $\s{C}$. Conversely, if $(\s{C},\otimes^{\s{C}},e^{\s{C}},\beta^{\s{C}})$ is a permutative category, then the Mac Lane coherence theorem implies that there is a unique operad map $\s{P} \to \b{End}(\s{C})$, which sends $\otimes$ to $\otimes^{\s{C}}$, $e$ to $e^{\s{C}}$, and $\beta$ to $\beta^{\s{C}}$. All told, there is a one-to-one correspondence
	\[
	\t{ev} : \{\t{$\s{P}$-algebras in $\b{Cat}$}\} \stackrel{\cong}{\to} \{\t{small permutative categories}\}
	\]
given by evaluation at $\otimes \in \s{P}(2)$, $e \in \s{P}(0)$, and $\beta \in \s{P}(2)$.

On the other hand, one can recover the coherence theorem from the bijectivity (or really the surjectivity) of evaluation. For suppose $(\s{C},\otimes^{\s{C}},e^{\s{C}},\beta^{\s{C}})$ is a permutative category, and let $\abs{\cdot}_{\s{C}} : \s{P} \to \b{End}(\s{C})$ be a lift of this structure to a $\s{P}$-action. Suppose further that we are given a diagram in $\s{C}$ such as
	\[
	\begin{tikzpicture}[scale=1.2]
		\node(A) at (0,0) {$a \otimes^{\s{C}} b \otimes^{\s{C}}c$};
		\node(B) at (2,1) {$a \otimes^{\s{C}} c \otimes^{\s{C}} b$};
		\node(C) at (6,1) {$c \otimes^{\s{C}} a \otimes^{\s{C}} b$};
		\node(D) at (8,0) {$c \otimes^{\s{C}} b \otimes^{\s{C}} a$};
		\node(E) at (4,-1) {$b \otimes^{\s{C}} a \otimes^{\s{C}} c$};
		
		\path[->]
		(A) edge [above left] node {$\t{id} \otimes^{\s{C}} \beta^{\s{C}}$} (B)
		(B) edge [above] node {$\beta^{\s{C}} \otimes^{\s{C}}\t{id}$} (C)
		(C) edge [above right] node {$\t{id} \otimes^{\s{C}} \beta^{\s{C}}$} (D)
		(A) edge [below left] node {$\beta^{\s{C}} \otimes \t{id}$} (E)
		(E) edge [below right] node {$\beta^{\s{C}}$} (D)
		;
	\end{tikzpicture}
	\]
In this case, there is an analogous diagram
	\[
	\begin{tikzpicture}[scale=1.2]
		\node(A) at (0,0) {$\t{id}_3$};
		\node(B) at (2,1) {$(23)$};
		\node(C) at (6,1) {$(123)$};
		\node(D) at (8,0) {$(13)$};
		\node(E) at (4,-1) {$(12)$};
		
		\path[->]
		(A) edge [above left] node {$\gamma(\otimes;\t{id},\beta)$} (B)
		(B) edge [above] node {$\gamma(\otimes;\beta,\t{id}) \cdot (23)$} (C)
		(C) edge [above right] node {$\gamma(\otimes;\t{id},\beta) \cdot (123)$} (D)
		(A) edge [below left] node {$\gamma(\otimes;\beta,\t{id})$} (E)
		(E) edge [below right] node {$\gamma(\beta;\otimes,\t{id}) \cdot (12)$} (D)
		;
	\end{tikzpicture}
	\]
in $\s{P}(3)$, whose image under $\abs{\cdot}_{\s{C}}$ and evaluation at $(a,b,c)$ is the previous diagram. Since every diagram in $\s{P}(3)$ commutes, so do their images, and hence the original diagram commutes. In general, every diagram in $\s{C}$ that lifts to $\s{P}$ must commute.

Thus, we think of the bijection between permutative categories and $\s{P}$-algebras as a reformulation of the classical coherence theorem. In what follows, we shall prove an analogous result for NSMCs.

\subsection{The coherence theorem for NSMCs} In order to state the coherence theorem for NSMCs, we must first introduce the corresponding operads. To that end, consider the category $\b{Sym}(\b{Set}^G)$ of symmetric sequences of $G$-sets, and let $\b{Op}(\b{Set}^G)$ be the category of symmetric operads of $G$-sets. By general considerations, there is a free-forgetful adjunction
	\[
	F : \b{Sym}(\b{Set}^G) \rightleftarrows \b{Op}(\b{Set}^G) : U.
	\]
We now define the operad $\c{SM}_{\c{T}}$ corresponding to $\c{T}$-NSMC structures.

\begin{defn}\label{defn:SMT} Given any set of exponents $\c{T} = (T_i)_{i \in I}$, where each $T_i$ is a finite, ordered $H_i$-set (cf. Definition \ref{defn:NSMC}), let
	\[
	S_{\c{T}} = \frac{G \times \Sigma_0}{G \times \{\t{id}_0\}} \sqcup \frac{G \times \Sigma_2}{G \times \{\t{id}_2\}} \sqcup \coprod_{i \in I} \frac{G \times \Sigma_{\abs{T_i}}}{\Gamma(T_i)} \in \b{Sym}(\b{Set}^G),
	\]
where $\Gamma(T_i)$ is the graph of the permutation representation corresponding to $T_i$ (cf. Construction \ref{const:TGCnorms}). Let $F(S_{\c{T}})$ be the free operad on $S_{\c{T}}$, and define the categorical operad $\c{SM}_{\c{T}}$ by
	\[
	\c{SM}_{\c{T}} := \til{F(S_{\c{T}})}.
	\]
By the considerations in \S\ref{sec:opalg2cats}, we obtain $2$-categories
	\[
	\c{SM}_{\c{T}}\t{-}\b{AlgSt} \subset \c{SM}_{\c{T}}\t{-}\b{AlgPs} \subset \c{SM}_{\c{T}}\t{-}\b{AlgLax}
	\]
of $\c{SM}_{\c{T}}$-algebras, equipped with $1$-morphisms of various strengths.

\end{defn}

Our next task is to define the evaluation functor $\t{ev} : \c{SM}_{\c{T}}\t{-}\b{AlgLax} \to \c{T}\b{SMLax}$. For that, we need some notation.

\begin{nota}\label{nota:eltFS} Let $\eta : S_{\c{T}} \to F(S_{\c{T}})$ be the unit of the adjunction. We write
	\[
	e = G \times \{\t{id}_0\}	\,,\,	\otimes = G \times \{\t{id}_2\}	\,,\,	\sbtn_{T_i} = \Gamma(T_i)	\, \in S_{\c{T}} ,
	\]
and we will sometimes write
	\[
	e = \eta(G \times \{\t{id}_0\})	\, , \,	\otimes = \eta(G \times \{\t{id}_2\})	\, , \,	\sbtn_{T_i} = \eta(\Gamma(T_i)) \, \in F(S_{\c{T}})
	\]
to reduce clutter.

Now let $\gamma$ denote operadic composition in $\c{SM}_{\c{T}}$, and let
	\begin{align*}
		\sbtn_0 &= e	\\
		\sbtn_1 &= \t{id}	\\
		\sbtn_n &= \underbrace{\gamma(\otimes;\gamma(\otimes; \cdots \gamma(\otimes;\otimes,\t{id}) \cdots , \t{id}),\t{id})}_{n - 1 \t{ copies of } \otimes}	\quad \t{if $n > 1$.}
	\end{align*}
We write
	\begin{align*}
		\alpha : \gamma(\otimes;\otimes,\t{id}) \to \gamma(\otimes;\t{id},\otimes)	\quad&\quad	\lambda: \gamma(\otimes;e,\t{id}) \to \t{id}	\\
		\beta : \tn  \to \tn \cdot (12)	\quad&\quad	\rho : \gamma(\otimes;\t{id},e) \to \t{id}	\\
		\upsilon_{T_i} : \sbtn_{T_i} \to \sbtn_{\abs{T_i}}	\quad
	\end{align*}
for the unique morphisms of $\c{SM}_{\c{T}}$ indicated.
\end{nota}

From here, we can define the evaluation map.

\begin{defn} Let $\c{T} = (T_i)_{i \in I}$ be a set of exponents, where each $T_i$ is a finite, ordered $H_i$-set. Suppose $\s{C}$ is a $\c{SM}_{\c{T}}$-algebra and let $\abs{\cdot} : \c{SM}_{\c{T}} \to \b{End}(\s{C})$ be its structure map. We can extract a $\c{T}$-normed symmetric monoidal structure on $\s{C}$ as follows. First, consider the values of $\abs{\cdot}$ on the generators of $\c{SM}_{\c{T}}$. Define $\otimes^{\s{C}} = \abs{\otimes}$, $e^{\s{C}} = \abs{e}$, \dots, $\alpha^{\s{C}} = \abs{\alpha}$, $\lambda^{\s{C}} = \abs{\lambda}$, \dots. We let $\t{ev}\s{C}$ denote the $G$-category $\s{C}$, equipped with these functors and natural transformations.

Next, suppose $(F,\vd_\bullet) : (\s{C},\abs{\cdot}_{\s{C}}) \to (\s{D},\abs{\cdot}_{\s{D}})$ is a lax $\c{SM}_{\c{T}}$-morphism between strict $\c{SM}_{\c{T}}$-algebras. Take $F_e := (\vd_0)_{e}$, $F_{\tn} := (\vd_2)_{\tn}$, and $F_{\btn_{T_i}} := (\vd_{|T_i|})_{\btn_{T_i}}$ for every $i \in I$. We let $\t{ev}F$ denote the $G$-functor $F$, together with these natural transformations.

Finally, suppose that $(\s{C},\abs{\cdot}_{\s{C}})$ and $(\s{D},\abs{\cdot}_{\s{D}})$ are strict $\c{SM}_{\c{T}}$-algebras, that $(F,\vd_\bullet)$ and $(F',\vd'_\bullet)$ are lax $\c{T}$-monoidal functors $\s{C} \to \s{D}$, and that $\omega : (F,\vd_\bullet) \Rightarrow (F',\vd_\bullet')$ is a $\c{SM}_{\c{T}}$-transformation between them. We set $\t{ev} \omega = \omega$.
\end{defn}

\begin{prop}Let $\c{T} = (T_i)_{i \in I}$ be a set of exponents, where each $T_i$ is a finite, ordered $H_i$-set. The evaluation map
	\[
	\t{ev} : \c{SM}_{\c{T}}\t{-}\b{AlgLax} \to \c{T}\b{SMLax}
	\]
is a $2$-functor.
\end{prop}

\begin{proof} We must show that $\t{ev}$ sends objects, $1$-morphisms, and $2$-morphisms in $\c{SM}_{\c{T}}\t{-}\b{AlgLax}$ to the corresponding data in $\c{T}\b{SMLax}$. From here, $2$-functoriality follows because composition and identities are defined in the same way.

We begin with objects. Suppose $\abs{\cdot} : \c{SM}_{\c{T}} \to \b{End}(\s{C})$ is a $\c{SM}_{\c{T}}$-algebra with underlying $G$-category $\s{C}$. Observe that every diagram in $\c{SM}_{\c{T}}$ commutes, and that $\abs{\cdot}$ is a map of operads in $G$-categories. Thus, it will be enough to locate the coherence diagrams for a NSMC in $\c{SM}_{\c{T}}$. Consider the pentagon axiom for $\tn$. This comes from the commutative pentagon
	\[
	\begin{tikzpicture}
		\node(1) at (0,0) {$\gamma(\otimes;\gamma(\otimes;\otimes,\t{id}),\t{id})$};
		\node(2) at (4,1) {$\gamma(\otimes;\otimes,\otimes)$};
		\node(3) at (8,0) {$\gamma(\otimes;\t{id},\gamma(\otimes;\t{id},\otimes))$};
		\node(4) at (1.5,-1.5) {$\gamma(\otimes;\gamma(\otimes;\t{id},\otimes),\t{id})$};
		\node(5) at (6.5,-1.5) {$\gamma(\otimes;\t{id},\gamma(\otimes;\otimes,\t{id}))$};
		\path[->]
		(1) edge [above left] node {$\gamma(\alpha;\otimes , \t{id} , \t{id})$} (2)
		(2) edge [above right] node {$\gamma(\alpha;\t{id},\t{id},\otimes)$} (3)
		(1) edge [left] node {$\gamma(\otimes ; \alpha , \t{id}) \quad$} (4)
		(4) edge [below] node {$\gamma(\alpha ; \t{id} ,\otimes , \t{id})$} (5)
		(5) edge [right] node {$\quad \gamma(\otimes ; \t{id} , \alpha)$} (3)
		;
	\end{tikzpicture}
	\]
in $\c{SM}_{\c{T}}(4)$. The remaining symmetric monoidal coherence diagrams axioms are visible in $\c{SM}_{\c{T}}(1) - \c{SM}_{\c{T}}(3)$. Lastly, twisted equivariance for $\ups_{T_i}^{\s{C}}$ can be deduced from a diagram in $\c{SM}_{\c{T}}(|T_i|)$. Given any index $i \in I$ and $h \in H_i$, there is a commutative diagram
	\[
	\begin{tikzpicture}
		\node(A) at (0,0) {$h \cdot \btn_{T_i} $};
		\node(B) at (4,0) {$\btn \!{}_{T_i} \cdot \sigma(h)$};
		\node(C) at (4,-1.5) {$\btn \!{}_{|T_i|} \cdot \sigma(h)$};
		\node(D) at (0,-3) {$h \cdot \btn \!{}_{|T_i|} $};
		\node(E) at (4,-3) {$\btn \!{}_{|T_i|}$};
		\path[->]
		(A) edge [above] node {$\id$} (B)
		(A) edge [left] node {$h \cdot \upsilon_{T_i}$} (D)
		(B) edge [right] node {$\upsilon_{T_i} \cdot \sigma(h)$} (C)
		(C) edge [right] node {$\t{composite of $\alpha$'s and $\beta$'s}$} (E)
		(D) edge [below] node {$\id$} (E)
		;
	\end{tikzpicture}
	\]
in $\c{SM}_{\c{T}}(|T_i|)$, where $(h,\sigma(h)) \in \Gamma(T_i)$. This maps to the twisted equivariance diagram in Definition \ref{defn:NSMC}, but with each $C_i$ replaced with $h^{-1} C_i$. It follows that $\t{ev}\s{C}$ is a $\c{T}$-NMSC.

Next, we consider $1$-morphisms. Suppose $(F , \partial_\bullet) : (\s{C},\abs{\cdot}_{\s{C}}) \to (\s{D},\abs{\cdot}_{\s{D}})$ is a lax $\c{SM}_{\c{T}}$-algebra morphism. Since $(\partial_n)_x$ is $G \times \Sigma_n$-equivariant in $x$ (axiom (ii)), it follows that $F_e = (\partial_0)_e : e^{\s{D}} \to F(e^{\s{C}})$ is $G$-fixed, $F_\otimes = (\partial_2)_{\otimes} : \otimes^{\s{D}} \circ F^{\times 2} \Rightarrow F \circ \otimes^{\s{C}}$ is $G$-natural, and for any $i \in I$, $F_{\btn_{T_i}} = (\vd_{|T_i|})_{\btn_{T_i}} : \btn_{T_i}^{\s{D}} \circ F^{\times \abs{T_i}} \Rightarrow F \circ \btn_{T_i}^{\s{C}}$ is $H_i$-natural. To get the lax $\c{T}$-monoidal diagrams relating coherence isomorphisms for $\s{C}$ and $\s{D}$, we use axioms (i), (iii), and (iv). For example, the naturality square for $\alpha \in \c{SM}_{\c{T}}(3)$ is
	\[
	\begin{tikzpicture}
		\node(1) at (0,0) {$\otimes^{\s{D}} \circ (\otimes^{\s{D}} \times \t{id}) \circ F^{\times 3}$};
		\node(2) at (0,-1.5) {$\otimes^{\s{D}} \circ F^{\times 2} \circ (\otimes^{\s{C}} \times \t{id})$};
		\node(3) at (0,-3) {$F \circ \otimes^{\s{C}} \circ (\otimes^{\s{C}} \times \t{id})$};
		\node(4) at (6,0) {$\otimes^{\s{D}} \circ (\t{id} \times \otimes^{\s{D}}) \circ F^{\times 3}$};
		\node(5) at (6,-1.5) {$\otimes^{\s{D}} \circ F^{\times 2} \circ (\t{id} \times \otimes^{\s{C}})$};
		\node(6) at (6,-3) {$F \circ \otimes^{\s{C}} \circ (\t{id} \times \otimes^{\s{C}})$};
		\path[->]
		(1) edge [above] node {$\alpha^{\s{D}} \circ F^{\times 3}$} (4)
		(3) edge [below] node {$F \circ \alpha^{\s{C}}$} (6)
		(1) edge [left] node {$\otimes^{\s{D}} \circ (F_\otimes \times \t{id})$} (2)
		(2) edge [left] node {$F_\otimes \circ (\otimes^{\s{C}} \times \t{id})$} (3)
		(4) edge [right] node {$\otimes^{\s{D}} \circ (\t{id} \times F_\otimes)$} (5)
		(5) edge [right] node {$F_\otimes \circ (\t{id} \times \otimes^{\s{C}})$} (6)
		;
	\end{tikzpicture}
	\]
and this is the usual relation for associators. The other relations are similar, which means $\t{ev}F = (F,F_\bullet) : \t{ev}\s{C} \to \t{ev}\s{D}$ is a lax $\c{T}$-monoidal functor.

Finally, we consider $2$-morphisms. If $\omega : (F,\vd_\bullet) \Rightarrow (F',\vd_\bullet')$ is a $\c{SM}_{\c{T}}$-transformation, then specializing the relation in Definition \ref{defn:Otransfm} to $x = \otimes$ and evaluating at $C,C' \in \s{C}$ gives the square below.
	\[
	\begin{tikzpicture}
		\node(1) at (0,0) {$FC \otimes^{\s{D}} FC'$};
		\node(2) at (0,-1.5) {$F(C \otimes^{\s{C}} C')$};
		\node(3) at (4,0) {$F'C \otimes^{\s{D}} F'C'$};
		\node(4) at (4,-1.5) {$F'(C \otimes^{\s{C}} C')$};
		\path[->]
		(1) edge [left] node {$F_\otimes$} (2)
		(3) edge [right] node {$F'_\otimes$} (4)
		(1) edge [above] node {$\omega \otimes^{\s{D}} \omega$} (3)
		(2) edge [below] node {$\omega$} (4)
		;
	\end{tikzpicture}
	\]
This is the usual relation between $F_\otimes$ and $F'_\otimes$. Similarly for the pairs $F_e$ and $F'_e$, and $F_{\btn_{T_i}}$ and $F'_{\btn_{T_i}}$. Therefore $\omega : (F,F_\bullet) \Rightarrow (F',F'_\bullet)$ is a $\c{T}$-transformation.
\end{proof}

Finally, we can state the coherence theorem for NSMCs.

\begin{thm}\label{thm:coh}Let $\c{T}$ be a set of exponents. There is a commutative triangle
	\[
	\begin{tikzpicture}
		\node(A) at (0,0) {$\c{SM}_{\c{T}}\t{-}\b{AlgLax}$};
		\node(B) at (4,0) {$\c{T}\b{SMLax}$};
		\node(C) at (2,-2) {$\ub{Cat}^G$};
		\path[->]
		(A) edge [above] node {$\t{ev}$} (B)
		(A) edge [below left] node {$\t{forget}$} (C)
		(B) edge [below right] node {$\t{forget}$} (C)
		;
	\end{tikzpicture}
	\]
of $2$-categories and $2$-functors, and the evaluation $2$-functor $\t{ev}$ is an isomorphism. Similarly in the strong and strict cases.
\end{thm}

\begin{proof} Combine Propositions \ref{prop:evbijobj}, \ref{prop:evbijfun}, and \ref{prop:evbijtrans} in \S\ref{sec:pfcoh}.
\end{proof}

As a consequence of this theorem, it follows that NSMCs model $N_\infty$ spaces. 

\begin{thm}\label{thm:NSMCmodel} Let $\c{T}$ be a set of exponents and let $\s{C}$ be a $\c{T}$-NSMC. Then the classifying space $B\s{C}$ is an algebra over the $N_\infty$ operad $B\c{SM}_{\c{T}}$. Moreover, the class of admissible sets of $B\c{SM}_{\c{T}}$ is the indexing system generated by $\c{T}$.
\end{thm}

We review a few definitions, and then we prove this result. Suppose $\s{O}$ is an operad in $G$-spaces. Then $\s{O}$ is a \emph{$N_\infty$ operad} if it is $\Sigma$-free, the subspace $\s{O}(n)^{\Gamma}$ is either empty or contractible for all subgroups $\Gamma \subset G \times \Sigma_n$, and all of the subspaces $\s{O}(n)^G$ are nonempty. We say that a finite $H$-set $T$ is an \emph{admissible set} of $\s{O}$ if $\s{O}(\abs{T})^{\Gamma(T)}$ is nonempty, where $\Gamma(T)$ denotes the graph subgroup corresponding to a permutation representation of $T$. The class of all admissible sets of a $N_\infty$ operad assemble into an \emph{indexing system}, i.e. a collection that contains all trivial actions, and is closed under isomorphism, conjugation, restriction, subobjects, coproducts, and self-induction (cf. \cite{BH}).

\begin{proof}[Proof of Theorem \ref{thm:NSMCmodel}] Suppose $\s{C}$ is $\c{T}$-NSMC. Then by Theorem \ref{thm:coh}, there is a $\c{SM}_{\c{T}}$-algebra structure on $\s{C}$ that evaluates to the given $\c{T}$-NSMC structure. Since the classifying space functor $B$ preserves finite products, it follows that $B\s{C}$ is an algebra over $B\c{SM}_{\c{T}}$.

The operad $F(S_{\c{T}}) = \t{Ob}(\c{SM}_{\c{T}})$ is a $N$ operad in the sense of \cite[Definition 3.1]{RubComb}, i.e. it is $\Sigma$-free and $F(S_{\c{T}})(n)^G \neq \varnothing$ for all $n$. The former assertion holds because $F(S_{\c{T}})$ maps into the $\Sigma$-free operad $\b{Set}(G,\b{As})$ and the latter holds because the unit induces a map $\eta : S_{\c{T}}(n)^G \to F(S_{\c{T}})(n)^G$ for every $n \geq 0$. Since $\til{(-)}$ sends every nonempty set to a contractible groupoid, the composite $B \circ \til{(-)}$ sends every nonempty set to a contractible space. Therefore $B\c{SM}_{\c{T}}$ is a $N_\infty$ operad.

Admissible sets for $N$ operads are defined exactly as for $N_\infty$ operads, and the functor $B \circ \til{(-)}$ preserves admissible sets (loc. cit. Proposition 3.5). Therefore $B\c{SM}_{\c{T}}$ has the same admissible sets as $F(S_{\c{T}})$. By \cite[Theorem 4.6]{RubComb}, the class of admissible sets of $F(S_{\c{T}})$ is the indexing system generated by the admissibles of $S_{\c{T}}$. This is the indexing system generated by $\c{T}$ because indexing systems are closed under conjugation and restriction.
\end{proof}

\section{The proof of the coherence theorem}\label{sec:pfcoh}

In this section, we prove Theorem \ref{thm:coh} by showing that the evaluation $2$-functor
	\[
	\t{ev} : \c{SM}_{\c{T}}\t{-}\b{AlgLax} \to \c{T}\b{SMLax}
	\]
is bijective on objects, $1$-morphisms, and $2$-morphisms. We begin by recalling a more precise description of $F(S_{\c{T}}) = \t{Ob}(\c{SM}_{\c{T}})$, which will be necessary for our analysis.

\subsection{The operad $\c{SM}_{\c{T}}$, again}\label{subsec:SMTdesc} In order to understand the structure of $\c{SM}_{\c{T}}$, we need to understand the combinatorics of its object operad, i.e. the free discrete operad $F(S_{\c{T}})$ on the $\Sigma$-free symmetric sequence $S_{\c{T}}$. As explained in \cite[\S5]{RubComb}, such operads have a description in terms of formal composites, which we now review.

Let $\c{T} = (T_i)_{i \in I}$ be a set of exponents, where each $T_i$ is a finite, ordered $H_i$-set. To start the construction of $\c{SM}_{\c{T}}$, choose a set $\{e = g_1^H , \dots , g_{\abs{G:H}}^H \}$ of $G/H$ coset representatives for each subgroup $H \subset G$. Next, consider the letters and punctuation symbols
	\begin{align*}
		x_n	&\quad(n=1,2,\dots)	\\
		e	&	\\
		\tn	&	\\
		r\sbtn_{T_i}	&\quad(\t{$i \in I$ and $r$ a $G/H_i$ coset representative})	\\
		(	\quad	)	\quad	,	&\quad(\t{punctuation})
	\end{align*}
A \emph{word} is a finite sequence $w = l_1 \dots l_n$ of the symbols above. A \emph{subword} of a word $w$ is a sequence that is either empty, or is of the form $l_j \dots l_k$ for $j \leq k$.

A \emph{term} is any word constructed by the following recursion
	\begin{enumerate}
		\item{}every variable $x_n$ is a term,
		\item{}the word $e()$ is a term,
		\item{}if $t_1,t_2$ are terms, then $\otimes(t_1,t_2)$ is a term, and
		\item{}if $t_1 , \dots , t_{\abs{T_i}}$ are terms, then $r\btn_{T_i}(t_1,\dots,t_{\abs{T_i}})$ is a term.
	\end{enumerate}
A \emph{subterm} of a term $t$ is a subword that is also a term. The \emph{arity} of a term $t$ is the number of distinct variable symbols that appear in $t$. A $n$-ary term $t$ is \emph{operadic} if each of the variables $x_1,\dots,x_n$ appears in $t$ exactly once.

\begin{defn}The $n$th level of the free operad $F(S_{\c{T}})$ is the set of all $n$-ary operadic terms. The operad $\c{SM}_{\c{T}}$ is obtained by applying the codiscrete functor $\til{(-)} : \b{Set} \to \b{Cat}$ levelwise.
\end{defn}

Next, we review the operad structure on $F(S_{\c{T}})$. If $t$ is a $n$-ary operadic term and $\sigma \in \Sigma_n$, the term $t \cdot \sigma$ is obtained by replacing each variable $x_i$ in $t$ with $x_{\sigma^{-1}i}$.

If $g \in G$, then $g * t$ is defined recursively (on all terms) by
	\begin{enumerate}
		\item{}$g * x_i = x_i$,
		\item{}$g * e() = e()$,
		\item{}$g * \otimes(t_1,t_2) = \otimes(g*t_1 , g*t_2)$, and
		\item{}$g * r\btn_{T_i}(t_1,\dots,t_{\abs{T_i}}) = r' \btn_{T_i}(g*t_{\sigma(h)^{-1}1} , \dots , g* t_{\sigma(h)^{-1}\abs{T_i}} )$, where $gr = r' h$ for a unique $G/H_i$ coset representative $r'$, $h \in H_i$, and $(h,\sigma(h)) \in \Gamma(T_i)$.
	\end{enumerate}
The point is that $g * (-)$ is supposed to be conjugation, so we should just multiply all function symbols in $t$ with $g$. However, $gr\bigotimes_{T_i}$ is not generally a letter in our alphabet, so we must modify the result.

The identity of $F(S_{\c{T}})$ is $x_1$.

The composite $\gamma(t;s_1,\dots,s_k)$ of $t \in F(S_{\c{T}})(k)$ with $s_1 \in F(S_{\c{T}})(j_1)$, \dots , $s_k \in F(S_{\c{T}})(j_k)$ is computed by
	\begin{enumerate}
		\item{}adding $j_1 + \cdots j_{i-1}$ to each subscript of $s_i$ -- call the result $s_i'$ -- and then
		\item{}substituting the terms $s_i'$ for the variables $x_i$ in $t$.
	\end{enumerate}
Here, the point is that substituting the $s_i$ directly leads to repeated variables, so we shift indices beforehand.

These data make $F(S_{\c{T}})$ into an operad in $\b{Set}^G$, which is free on $S_{\c{T}}$. Recalling Notation \ref{nota:eltFS}, the unit $\eta : S_{\c{T}} \to F(S_{\c{T}})$ sends $e$ to $e()$, $\otimes$ to $\otimes(x_1,x_2)$, and $\btn_{T_i}$ to $\btn_{T_i}(x_1,\dots,x_{\abs{T}})$. The rest is determined by equivariance.

Now we can begin the proof of Theorem \ref{thm:coh}.

\subsection{Bijectivity of $\t{ev}$ objects}

\begin{lem}\label{lem:evobjinj} The evaluation $2$-functor is injective on objects.
\end{lem}

\begin{proof}Suppose $\abs{\cdot}_1 , \abs{\cdot}_2 : \c{SM}_{\c{T}} \to \b{End}(\s{C})$ are $\c{SM}_{\c{T}}$-algebra structures on $\s{C}$ that evaluate to the same NSMC structure. Then $\abs{\cdot}_1$ and $\abs{\cdot}_2$ agree on $e()$, $\otimes(x_1,x_2)$, and $\btn_{T_i}(x_1,\dots,x_{\abs{T_i}}) \in \c{SM}_{\c{T}}$ by assumption, and hence they agree on all objects of $\c{SM}_{\c{T}}$ because $\t{Ob}(\c{SM}_{\c{T}}) = F(S_{\c{T}})$ is free on these operations. Similarly, $\abs{\cdot}_1$ and $\abs{\cdot}_2$ agree on $\alpha,\lambda,\rho,\beta,\upsilon_{T_i} \in \c{SM}_{\c{T}}$ by assumption, and hence they agree on all morphisms of $\c{SM}_{\c{T}}$, because the preceding morphisms and their inverses generate all morphisms of $\c{SM}_{\c{T}}$ under the operad and category structures. 
\end{proof}

Now for surjectivity. Given a $\c{T}$-NSMC
	\[
	\Big( \s{C},\tn^{\s{C}},e^{\s{C}},(\sbtn_{T_i}^{\s{C}})_{i \in I} , \alpha^{\s{C}},\lambda^{\s{C}},\rho^{\s{C}},\beta^{\s{C}},(\ups^{\s{C}}_{T_i})_{i \in I} \Big) ,
	\]
we must find an operad map $\abs{\cdot} : \c{SM}_{\c{T}} \to \b{End}(\s{C})$ that evaluates back to these generating data. We follow Mac Lane \cite[Ch. VII.2]{CWM} closely and encourage the reader to review those arguments before proceeding.

The main difference between our situation and Mac Lane's is the presence of external norms and untwistors, but these additional data barely interact with the ordinary symmetric monoidal structure. Thus, our strategy is to separate out the new data from the old, thus reducing to Mac Lane's classical theorem. This is similar to how one deduces joint coherence for associators and unitors from the coherence of associators alone.

The remainder of this section gives a three-stage construction of the necessary operad map $\abs{\cdot} : \c{SM}_{\c{T}} \to \b{End}(\s{C})$.

\subsubsection{Step 1:} Recall Notation \ref{nota:eltFS}. We define a map
	\[
	S_{\c{T}} \to \t{Ob}(\b{End}(\s{C}))
	\]
of symmetric sequences by sending each symbol $e, \otimes, \btn_{T_i} \in S_{\c{T}}$ to the corresponding operations $e^{\s{C}}() , \otimes^{\s{C}}(x_1,x_2) , \btn^{\s{C}}_{T_i}(x_1,\dots,x_{\abs{T_i}}) \in \b{End}(\s{C})$. By adjunction, this extends freely to an operad map
	\[
	\abs{\cdot}_1 : \t{Ob}(\c{SM}_{\c{T}}) = F(S_{\c{T}}) \to \t{Ob}(\b{End}(\s{C})),
	\]
which sends $e(), \otimes(x_1,x_2), \btn_{T_i}(x_1,\dots,x_{\abs{T_i}}) \in F(S_{\c{T}})$ to the corresponding operations on $\s{C}$.

\subsubsection{Step 2:}\label{subsubsec:step2const} Next, let $n \geq 0$, and consider the set map
	\[
	\abs{\cdot}_1 : F(S_{\c{T}})(n) \to \t{Ob}(\b{Cat}_G(\s{C}^{\times n} , \s{C})).
	\]
We shall extend this to a graph homomorphism, and then to a functor.

\begin{defn}\label{defn:basicedge} Let $\c{T} = (T_i)_{i \in I}$ be a set of exponents. An \emph{$\eta$-basic edge} is a tuple
	\[
	b = (t,j,\eta,s_1,\dots,s_k,\sigma)
	\]
such that
	\begin{enumerate}
		\item{} $t$ is an operadic term of positive arity,
		\item{} $j$ is a natural number between $1$ and $\t{arity}(t)$,
		\item{}$\eta = \alpha^{\pm 1}$ or $\lambda^{\pm 1}$ or $\rho^{\pm 1}$ or $\beta$ or $r\ups_{T_i}^{\pm 1}$ for some index $i \in I$ and $G/H_i$ coset representative $r$ (cf. Notation \ref{nota:eltFS}),
		\item{}$s_1, \dots, s_k$ are operadic terms, where $k$ is the arity of $\eta$, and
		\item{}$\sigma$ is a permutation in $\Sigma_N$, where $N = \t{arity}(t) + \t{arity}(s_1) + \cdots + \t{arity}(s_k) - 1$.
	\end{enumerate}
We call the number $N$ in (5) the \emph{arity} of $b$, and we shall usually shorten ``$r\ups_{T_i}^{\pm 1}$-basic edge''  to ``$\ups^{\pm 1}$-basic edge.''

Let $\gamma$ denote operadic composition and let $\circ_j$ denote partial composition. The \emph{source} and \emph{target} of the basic edge $b = (t,j,\eta,s_1,\dots,s_k,\sigma)$ are the operadic terms
	\[
	t \circ_j \gamma(\t{dom}(\eta);s_1,\dots,s_k) \cdot \sigma \quad\t{and}\quad t \circ_j \gamma(\t{cod}(\eta);s_1,\dots,s_k) \cdot \sigma
	\]
in $F(S_{\c{T}})$, and the \emph{interpretation} of $b$ is the natural isomorphism
	\[
	\abs{b}_2 := \abs{t}_1 \circ_j \gamma(\eta^{\s{C}}; \abs{s_1}_1 , \dots , \abs{s_k}_1 ) \cdot \sigma : \abs{\t{source}(b)} \Rightarrow \abs{\t{target(b)}}
	\]
in the endomorphism operad $\b{End}(\s{C})$.
\end{defn}
	
\begin{nota} Let $\b{Bas}(n)$ be the directed graph, whose vertex set is $F(S_{\c{T}})(n)$, and whose edges are the $n$-ary basic edges. Let $\b{Fr}(\b{Bas}(n))$ be the free category generated by $\b{Bas}(n)$.
\end{nota}

We extend the set map $\abs{\cdot}_1 : F(S_{\c{T}})(n) \to \t{Ob}(\b{Cat}_G(\s{C}^{\times n} , \s{C}))$ to a graph homomorphism
	\[
	\b{Bas}(n) \to \b{Cat}_G(\s{C}^{\times n} , \s{C})
	\]
by sending each basic edge to its interpretation in $\b{End}(\s{C})$. This graph homomorphism freely extends to a functor
	\[
	\abs{\cdot}_2 : \b{Fr}(\b{Bas}(n)) \to \b{Cat}_G(\s{C}^{\times n} , \s{C}).
	\]

\subsubsection{Step 3:} For any objects $x,y \in \b{Fr}(\b{Bas}(n))$, there are morphisms $x,y \rightrightarrows \otimes( \cdots \otimes(\otimes(x_1,x_2),x_3) \cdots, x_n)$, and hence a morphism $x \to y$. Therefore the quotient of $\b{Fr}(\b{Bas}(n))$ by the congruence relation
	\[
	p \sim q	\quad\t{if and only if}\quad	p \t{ and } q \t{ have the same domain and codomain}
	\]
is isomorphic to the category $\c{SM}_{\c{T}}(n)$. We shall prove that
	\[
	\abs{\cdot}_2 : \b{Fr}(\b{Bas}(n)) \to \b{Cat}_G(\s{C}^{\times n} , \s{C})
	\]
descends to a functor
	\[
	\abs{\cdot} = \abs{\cdot}_3 : \c{SM}_{\c{T}}(n) \cong \b{Fr}(\b{Bas}(n)) / \!\! \sim
	\,\, \to
	\b{Cat}_G(\s{C}^{\times n} , \s{C})
	\]
for every $n \geq 0$. This is one form of the coherence theorem.

\begin{lem}\label{lem:nubasicsw}Suppose that $t_1 \stackrel{e}{\to} t_2 \stackrel{u}{\to} t_3$ are basic edges in $\b{Bas}(n)$, and
	\begin{enumerate}[label=(\alph*)]
		\item{}the edge $e$ is $\ve$-basic, where $\ve$ is one of $\alpha^{\pm 1}$, $\lambda^{\pm 1}$, $\rho^{\pm 1}$, or $\beta$, and
		\item{}the edge $u$ is $\upsilon$-basic.
	\end{enumerate}
Then there is a composable pair of basic edges $t_1 \stackrel{u'}{\to} t_2' \stackrel{e'}{\to} t_3$ such that
	\begin{enumerate}[label=(\roman*)]
		\item{}the edge $e'$ is $\ve$-basic,
		\item{}the edge $u'$ is $\upsilon$-basic, and
		\item{} $\abs{e'}_2 \vc \abs{u'}_2 = \abs{u}_2 \vc \abs{e}_2$.
	\end{enumerate}
In (iii), $\bullet$ denotes (vertical) composition in $\b{End}(\s{C})$.
\end{lem}

For example, given a binary external norm $\boxtimes$ and basic edges
	\[
	\boxtimes(x_2,\otimes(x_1,e())) \stackrel{e}{\to} \boxtimes(x_2,x_1) \stackrel{u}{\to} \otimes(x_2,x_1),
	\]
the interchanged edges are
	\[
	\boxtimes(x_2,\otimes(x_1,e())) \stackrel{u'}{\to} \otimes(x_2,\otimes(x_1,e())) \stackrel{e'}{\to} \otimes(x_2,x_1),
	\]
and the equation
	\[
	(x_2 \otimes \rho_{x_1}) \circ \ups_{x_2, x_1 \otimes e} = \ups_{x_2,x_1} \circ (x_2 \boxtimes \rho_{x_1})
	\]
holds by the naturality of $\ups$.

\begin{proof}[Proof of of Lemma \ref{lem:nubasicsw}]The edge $e$ replaces some subterm of $t_1$ of the form
	\[
	\otimes(\otimes(s_1,s_2),s_3) \t{ or } \otimes(e(),s_1) \t{ or } \otimes(s_1,e()) \t{ or } \dots
	\]
with another subterm. The edge $u$ replaces some subterm of $t_2$ of the form
	\[
	r \btnt(s_1,\dots,s_{\abs{T_i}})
	\]
with $\otimes(\cdots \otimes ( \otimes ( s_1 , s_2 ), s_3) \cdots s_{\abs{T_i}})$. These substitutions commute, and switching the order yields the edges $u'$ and $e'$. The composite $\abs{u}_2 \bullet \abs{e}_2$ is obtained by pasting two disjoint natural transformations on to $\abs{t_1}_2$. The composite $\abs{e'}_2 \bullet \abs{u'}_2$ is obtained by pasting the same two natural transformations in the opposite order. Therefore $\abs{u}_2 \bullet \abs{e}_2 = \abs{e'}_2 \bullet \abs{u'}_2$.
\end{proof}

Next, we consider a special case of the coherence theorem. 

\begin{defn} A morphism $p$ in $\b{Fr}(\b{Bas}(n))$ is \emph{$\ups$-directed} if $p$ decomposes into a composite $p = b_k \vc \cdots \vc b_1$ of basic edges, none of which are $\ups^{-1}$-basic. Here $\bullet$ denotes composition in $\b{Fr}(\b{Bas}(n))$.
\end{defn}

Recall that $\btn_n(x_1,\dots,x_n)$ is shorthand for $ \tn(\cdots\tn(\tn(x_1,x_2),x_3) \cdots , x_n)$ (cf. Notation \ref{nota:eltFS}).

\begin{lem}\label{lem:nudir}Suppose that $t \in \b{Fr}(\b{Bas}(n))$, and that $p , q : t \rightrightarrows \btn_n(x_1,\dots,x_n)$ are $\upsilon$-directed. Then the natural transformations $\abs{p}_2 , \abs{q}_2 : \abs{t}_2 \rightrightarrows \abs{\btn_n(x_1,\dots,x_n)}_2$ are equal in $\b{End}(\s{C})$.
\end{lem}

\begin{proof}Consider $p$ alone first. We may write $p = b_k \vc \cdots \vc b_1$ for unique basic edges, none of which are $\ups^{-1}$-basic, and then $\abs{p}_2$ is equal to the vertical composite $\abs{b_k}_2 \vc \cdots \vc \abs{b_1}_2$. By applying Lemma \ref{lem:nubasicsw} repeatedly, we may move all images of $\upsilon$-basic edges to the right. Let $t^{\t{red}}$ be the term obtained by replacing every subterm $s \subset t$ of the form $s = r \btn_T(s_1,\dots,s_{|T_i|})$ with the term $\tn(\cdots \tn(\tn(s_1,s_2),s_3) \cdots s_{|T_i|})$. Then the above interchange yields morphisms $p_{\ups} : t \to t^{\t{red}}$ and $p_s : t^{\t{red}} \to \btn_n(x_1,\dots,x_n)$ in $\b{Fr}(\b{Bas}(n))$ such that
	\begin{enumerate}[label=(\alph*)]
		\item{}$\abs{p}_2 = \abs{p_s}_2 \vc \abs{p_\upsilon}_2$,
		\item{}$p_\upsilon$ is a composite of $\upsilon$-basic edges only, and
		\item{}$p_s$ is a composite of ordinary symmetric monoidal basic edges.
	\end{enumerate}
Now do the same thing for $q$. We obtain parallel pairs of maps $p_\upsilon , q_\upsilon : t \rightrightarrows t^{\t{red}}$ and $p_s , q_s : t^{\t{red}} \rightrightarrows \btn_n (x_1 , \dots, x_n)$ such that $\abs{p}_2 = \abs{p_s}_2 \vc \abs{p_\upsilon}_2$ and $\abs{q}_2 = \abs{q_s}_2 \vc \abs{q_\upsilon}_2$.

The ordinary Mac Lane coherence theorem implies that $\abs{p_s}_2 = \abs{q_s}_2$, since $\abs{p_s}_2$ and $\abs{q_s}_2$ come from the underlying symmetric monoidal structure on $\s{C}$. For $p_\ups$ and $q_\ups$, write $p_\ups = u_n \bullet \cdots \bullet u_1$ where each $u_i$ is $\upsilon$-basic. Each $u_i$ replaces a single instance of $r \btn_T(s_1,\dots,s_{\abs{T_i}})$ in $t$ with $\tn(\cdots \tn(\tn(s_1,s_2),s_3) \cdots s_{|T_i|})$. Therefore the composite $\abs{p_\ups} = \abs{u_n}_2 \bullet \cdots \bullet \abs{u_1}_2$ is obtained by pasting $n$ disjoint untwistors on to $\abs{t}_2$. We may paste these transformations in any order without changing the value of $\abs{p_\ups}_2$. On the other hand, $\abs{q_\ups}_2$ is obtained by pasting the same $n$ untwistors on to $\abs{t}_2$, possibly in a different order, and hence $\abs{p_\ups}_2 = \abs{q_\ups}_2$. 
\end{proof}

Now for the general case.

\begin{prop} Suppose that $t , t' \in \b{Fr}(\b{Bas}(n))$ and that $p : t \to t'$ is arbitrary. Choose $\upsilon$-directed morphisms $d : t \to \btn_n (x_1,\dots,x_n)$ and $d' : t' \to \btn_n (x_1,\dots,x_n)$. Then $\abs{p}_2 = \abs{d'}^{-1}_2 \vc \abs{d}_2 : \abs{t}_2 \to \abs{\btn_n(x_1,\dots,x_n)}_2 \to \abs{t'}_2$, where $\bullet$ denotes vertical composition.
\end{prop}

\begin{proof} Consider  $p = b_k \vc \cdots \vc b_1 : t \to t'$, where the $b_i$ are basic edges. By regrouping the factors, we may write
	\[
	p = \Big[ t = t_0 \stackrel{s_0}{\to} t_1 \stackrel{u_1}{\to} t_2 \stackrel{s_1}{\to} t_3 \stackrel{u_2}{\to} t_4 \to \cdots \to t_{2j-1} \stackrel{u_j}{\to} t_{2j} \stackrel{s_j}{\to} t_{2j+1} = t' \Big]
	\]
where each $u_i$ is $\upsilon^{\pm 1}$-basic, and no instances of $\upsilon^{\pm 1}$-basic edges occur in the $s_i$. Applying $\abs{\cdot}_2$ and replacing each $\upsilon^{-1}$-basic edge with its reverse, we see
	\[
	\abs{p}_2 = \Big[ \abs{t_0}_2 \stackrel{\abs{s_0}_2}{\to} \abs{t_1}_2 \stackrel{\abs{u_1}^{\pm 1}_2}{\to} \abs{t_2}_2 \to \cdots \to \abs{t_{2j-1}}_2 \stackrel{\abs{u_j}^{\pm 1}_2}{\to} \abs{t_{2j}}_2 \stackrel{\abs{s_j}_2}{\to} \abs{t_{2j+1}}_2 \Big]
	\]
where each of the $u_i$ are now $\upsilon$-basic edges, possibly pointing backwards. Now, for each $t_i$, choose a $\upsilon$-directed morphism $d_i : t_i \to \btn_n (x_1,\dots,x_n)$, taking $d_0 = d$ and $d_{2j + 1} = d'$ from the theorem statement. We obtain a diagram
	\begin{center}
		\begin{tikzpicture}[scale=1.4]
			\node(0) {$\abs{t_0}_2$};
			\node(1) at (1,0) {$\abs{t_1}_2$};
			\node(2) at (2,0) {$\abs{t_2}_2$};
			\node(3) at (3,0) {$\abs{t_3}_2$};
			\node(dots) at (4,0) {$\cdots$};
			\node(2j-1) at (5.25,0) {$\abs{t_{2j-1}}_2$};
			\node(2j) at (6.6,0) {$\abs{t_{2j}}_2$};
			\node(2j+1) at (8,0) {$\abs{t_{2j+1}}_2$};
			\node(r) at (4,-1.5) {$\abs{\btn_n(x_1,\dots,x_n)}_2$};
			\node(dots') at (4,-0.4) {$\cdots$};
			\path[->]
			(0) edge [above] node {$\abs{s_0}_2$} (1)
			(1) edge [above] node {$\abs{u_1}^{\pm 1}_2$} (2)
			(2) edge [above] node {$\abs{s_1}_2$} (3)
			(3) edge node {} (dots)
			(dots) edge node {} (2j-1)
			(2j-1) edge [above] node {$\abs{u_j}^{\pm 1}_2$} (2j)
			(2j) edge [above] node {$\abs{s_j}_2$} (2j+1)
			
			(0) edge node {} (r)
			(1) edge node {} (r)
			(2) edge node {} (r)
			(3) edge node {} (r)
			(2j-1) edge node {} (r)
			(2j) edge node {} (r)
			(2j+1) edge node {} (r)
			;
		\end{tikzpicture}
	\end{center}
in $\b{End}(\s{C})(n)$, where the diagonal map $\abs{t_i}_2 \to \abs{\btn_n (x_1,\dots,x_n)}_2$ is $\abs{d_i}_2$. Every triangle commutes by Lemma \ref{lem:nudir}, and hence $\abs{p}_2$ is equal to the vertical composite $\abs{d_{2j+1}}^{-1}_2 \vc \abs{d_0}_2 = \abs{d'}^{-1}_2 \vc \abs{d}_2$.
\end{proof}

It follows that for any two $p,q : t \rightrightarrows t'$ in $\b{Fr}(\b{Bas}(n))$, we have $\abs{p}_2 = \abs{q}_2$, and the next statement follows.

\begin{cor} The functor $\abs{\cdot}_2 : \b{Fr}(\b{Bas}(n)) \to \b{End}(\s{C})(n)$ factors through the quotient $\pi: \b{Fr}(\b{Bas}(n)) \to \b{Fr}(\b{Bas}(n)) / \langle p \sim q \, | \, p, \, q \t{ parallel} \rangle$.
\end{cor}

\begin{defn}\label{defn:abs3} For every $n \geq 0$, let
	\[
	\abs{\cdot} = \abs{\cdot}_3 : \c{SM}_{\c{T}}(n) \to \b{End}(\s{C})(n)
	\]
be the functor induced by $\abs{\cdot}_2 : \b{Fr}(\b{Bas}(n)) \to \b{End}(\s{C})(n)$.
\end{defn}

\begin{lem}\label{lem:evsurjobj} The functors $\abs{\cdot} : \c{SM}_{\c{T}}(n) \to \b{End}(\s{C})(n)$ define an operad map, which evaluates to the NSMC structure on $\s{C}$.
\end{lem}

\begin{proof} The functors $\abs{\cdot}$ evaluate to the NSMC structure on $\s{C}$ by construction. We need to show that $\abs{\cdot}$ is an operad map. Since $\abs{\cdot} = \abs{\cdot}_3$ extends the operad map $\abs{\cdot}_1 : F(S_{\c{T}}) \to \b{End}(\s{C})$, this is immediate on the level of objects. It remains to check on morphisms.

First, we consider operadic identities. The map $\abs{\cdot}_1 : F(S_{\c{T}}) \to \t{Ob}(\b{End}(\s{C}))$ is an operad map. Therefore it preserves the operadic identity object, and therefore its functorial extension to $\abs{\cdot} : \c{SM}_{\c{T}} \to \b{End}(\s{C})$ preserves the identity morphism on the operadic identity.

Next, we consider $\Sigma$-equivariance. By functoriality, it will be enough to show that $\abs{\ol{b}\tau} = \abs{\ol{b}}\tau$ for every basic edge $b = (t,j,\eta,s_1,\dots,s_k,\sigma)$ in $\b{Fr}(\b{Bas})$ with residue class $\ol{b}$ in $\c{SM}_{\c{T}}$. Let $c = (t,j,\eta,s_1,\dots,s_k,\sigma\tau)$. Then
	\[
	\abs{c}_2 = \abs{t}_1 \circ_j \gamma(\eta^{\s{C}} ; \abs{s_1}_1,\dots,\abs{s_k}_1) \cdot \sigma\tau = \abs{b}_2 \tau = \abs{\ol{b}}\tau.
	\]
Moreover,
	\[
	\t{source}(c) = t \circ_j \gamma(\t{dom}(\eta);s_1,\dots,s_k) \cdot \sigma\tau = \t{source}(b)\tau
	\]
and $\t{target}(c) = \t{target}(b)\tau$ for similar reasons. Therefore $\ol{c}$ and $\ol{b}\tau$ are parallel in $\c{SM}_{\c{T}}$, and hence equal. Consequently, $\abs{\ol{b}\tau} = \abs{\ol{c}} = \abs{c}_2 = \abs{\ol{b}}\tau$.

Next, we consider $G$-equivariance. As above, it will be enough to show $\abs{g \ol{b}} = g \abs{\ol{b}}$ for every basic edge $b = (t,j,\eta,s_1,\dots,s_k,\sigma)$. There are two cases: either $\eta = r\ups_{T_i}^{\pm 1}$, or it does not. In the latter case, let $c = (gt,j,\eta,gs_1,\dots,gs_k,\sigma)$. Then
	\[
	\abs{c}_2 = \abs{gt}_1 \circ_j \gamma(\eta^{\s{C}} ; \abs{gs_1}_1 , \dots , \abs{gs_k}_1) \cdot \sigma =  g \cdot \Big( \abs{t}_1 \circ_j \gamma(\eta^{\s{C}} ; \abs{s_1}_1 , \dots , \abs{s_k}_1) \cdot \sigma \Big) = g \cdot \abs{\ol{b}}
	\]
because $\abs{\cdot}_1$ is an operad map and $\eta^{\s{C}}$ is $G$-fixed. As before, $\ol{c} = g\ol{b}$ because they are parallel, and therefore $\abs{g \ol{b}} = \abs{\ol{c}} = g \abs{\ol{b}}$.

Now suppose $\eta = r\ups_{T_i}$. The case where $\eta = r\ups_{T_i}^{-1}$ is similar. If $\eta = r\ups_{T_i}$, then
	\begin{align*}
		g \abs{\ol{b}} = \abs{gt}_1 \circ_j \gamma (g r\ups_{T_i}^{\s{C}} ; \abs{gs_1}_1,\dots,\abs{gs_k}_1) \cdot \sigma .
	\end{align*}
Writing $gr = r'h$ for a unique $G/H_i$ coset representative $r'$ and $h \in H$, the twisted equivariance of $\ups^{\s{C}}_{T_i}$ gives $gr\ups^{\s{C}}_{T_i} = \xi(h)^{-1}_{\s{C}} \bullet (r' \ups^{\s{C}}_{T_i} \cdot \xi(h))$, where $\xi(h)^{-1}_{\s{C}}$ is the symmetric monoidal coherence map that permutes the factors of $\otimes_{\abs{T_i}}(x_1,\dots,x_{\abs{T_i}})$ by $\xi(h)^{-1}$, and $(h,\xi(h))$ is an element of the graph subgroup $\Gamma(T_i)$. Therefore $g \abs{\ol{b}}$ is equal to
	\[
	\Big( \abs{gt}_1 \circ_j \gamma ( \xi(h)^{-1}_{\s{C}} ; \abs{gs_1}_1,\dots,\abs{gs_k}_1) \cdot \sigma \Big) \bullet \Big( \abs{gt}_1 \circ_j \gamma (r' \ups_{T_i}^{\s{C}} \cdot \xi(h) ; \abs{gs_1}_1,\dots,\abs{gs_k}_1) \cdot \sigma \Big).
	\]
Breaking $\xi(h)^{-1}_{\s{C}}$ apart in the left factor breaks that factor apart into the images of basic edges, say $b_1, \dots, b_n$, and pulling $\xi(h)$ out of the right factor makes that factor into the image of a basic edge $c$. Let $d = b_1 \bullet \cdots \bullet b_n \bullet c$. Then $\abs{g \ol{b}} = \abs{\ol{d}} = g\abs{\ol{b}}$.

Lastly, we consider composition. Suppose $\ol{p} : s \to s' \in \c{SM}_{\c{T}}(k)$ and $\ol{q_i} : t_i \to t_i' \in \c{SM}_{\c{T}}(j_i)$ are congruence classes of morphisms $p,q_i \in \b{Fr}(\b{Bas})$, for $i = 1 , \dots ,k$. We may write $p = b_1 \vc \cdots \vc b_n$ and $q_j = c_{j1} \vc \cdots \vc c_{j n_j}$, where each $b_i$ and $c_{jk}$ is basic, and then $\gamma(\abs{\ol{p}};\abs{\ol{q_1}},\dots,\abs{\ol{q_k}})$ factors as
	\[
	\gamma(\abs{\ol{b_1}};\t{id}) \vc \cdots \vc \gamma(\abs{\ol{b_n}};\t{id}) \vc \gamma(\t{id};\abs{\ol{c_{11}}},\t{id}) \vc \cdots \vc \gamma(\t{id};\abs{\ol{c_{1 n_1}}},\t{id}) \vc \cdots \vc \gamma(\t{id};\t{id},\dots,\abs{\ol{c_{k n_k}}})
	\]
Each of these factors lifts to a basic edge in $\b{Fr}(\b{Bas})$, and composing all of these lifts together yields a morphism
	\[
	d : \gamma(s;t_1,\dots,t_k) \to \gamma(s';t_1',\dots,t_k').
	\]
Then $\abs{\gamma(\ol{p},\ol{q_1},\dots,\ol{q_k})} = \abs{\ol{d}} = \gamma(\abs{\ol{p}};\abs{\ol{q_1}},\dots,\abs{\ol{q_k}})$.
\end{proof}

\begin{prop}\label{prop:evbijobj} The $2$-functor $\t{ev} : \c{SM}_{\c{T}}\t{-}\b{AlgLax} \to \c{T}\b{SMLax}$ is bijective on objects, and similarly in the pseudo and strict cases.
\end{prop}

\begin{proof} Combine Lemmas \ref{lem:evobjinj} and \ref{lem:evsurjobj}.
\end{proof}

\subsection{Bijectivity on morphisms} The bijectivity of the evaluation $2$-functor $\t{ev} : \c{SM}_{\c{T}}\t{-}\b{AlgLax} \to \c{T}\b{SMLax}$ on $1$-morphisms and $2$-morphisms is comparatively simple. We begin with $1$-morphisms.

\begin{defn}Let $G$ be a finite group, and suppose $F,H : \s{C} \rightrightarrows \s{D}$ are a pair of $G$-functors between $G$-categories. A \emph{(not necessarily natural) $G$-transformation $F \not\Rightarrow H$} is a family of morphisms
	\[
	\eta = (\eta_c : Fc \to Hc)_{c \in \s{C}}.
	\]
such that $g\eta_c = \eta_{gc}$ for all $c \in \s{C}$ and $g \in G$. Given any morphism $f : c \to c'$ in $\s{C}$, we say that $\eta$ is \emph{$f$-natural} if the square
	\[
	\begin{tikzpicture}
		\node(A) at (0,0) {$Fc$};
		\node(B) at (2,0) {$Hc$};
		\node(C) at (0,-2) {$Fc'$};
		\node(D) at (2,-2) {$Hc'$};
		\path[->]
		(A) edge [above] node {$\eta_c$} (B)
		(A) edge [left] node {$Ff$} (C)
		(B) edge [right] node {$Hf$} (D)
		(C) edge [below] node {$\eta_{c'}$} (D)
		;
	\end{tikzpicture}
	\]
commutes.
\end{defn}

\begin{lem}\label{lem:fnatGcat} Suppose $F , H : \s{C} \rightrightarrows \s{D}$ are $G$-functors and $\eta$ is a (not necessarily natural) $G$-transformation from $F$ to $H$. Then the morphisms $f \in \s{C}$ for which $\eta$ is $f$-natural form a $G$-subcategory of $\s{C}$, and $\eta$ is a $G$-natural transformation if and only if this subcategory is all of $\s{C}$.
\end{lem}

\begin{proof}The transformation is $\t{id}_c$-natural for every $c \in \s{C}$, and the collection of all $f$ for which $\eta$ is $f$-natural is closed under composition and the $G$-action.
\end{proof}

The notion of $f$-naturality also interacts well with operadic composition.

\begin{lem}\label{lem:fnatop} Let $\s{O}$ be an operad in $G$-categories and let $\abs{\cdot}_{\s{C}} : \s{O} \to \b{End}(\s{C})$ and $\abs{\cdot}_{\s{D}} : \s{O} \to \b{End}(\s{D})$ be $\s{O}$-algebras. Suppose:
	\begin{enumerate}
		\item{}$F : \s{C} \to \s{D}$ is a $G$-functor,
		\item{}$\vd_n : \abs{\cdot}_{\s{D}} \circ F^{\times n} \not\Rightarrow F \circ \abs{\cdot}_{\s{C}}$ is a not necessarily natural transformation for all $n \geq 0$, and
		\item{}condition (iv) of Definition \ref{defn:laxopmap} holds.
	\end{enumerate}
Given any morphisms $g : y \to y' \in \s{O}(k)$ and $f_i : x_i \to x_i' \in \s{O}(j_i)$ for $i = 1 , \dots, k$, if $\vd_k$ is $g$-natural and $\vd_{j_i}$ is $f_i$-natural for all $i$, then the transformation $\vd_{j_1 + \cdots + j_k}$ is $\gamma(g;f_1,\dots,f_k)$-natural.
\end{lem}

\begin{proof}The $\gamma(g;f_1,\dots,f_k)$-naturality square for $\vd_{j_1 + \cdots + j_k}$ factors as:	\[
	\begin{tikzpicture}
		\node(11) at (0,0) {$\abs{y}_{\s{D}} \circ (\abs{x_1}_{\s{D}} \times \cdots \times \abs{x_k}_{\s{D}}) \circ F^{\times \sum j_i}$};
		\node(12) at (7,0) {$\abs{y'}_{\s{D}} \circ (\abs{x_1'}_{\s{D}} \times \cdots \times \abs{x_k'}_{\s{D}}) \circ F^{\times \sum j_i}$};
		\node(21) at (0,-1.5) {$\abs{y}_{\s{D}} \circ F^{\times k} \circ (\abs{x_1}_{\s{C}} \times \cdots \times \abs{x_k}_{\s{C}})$};
		\node(22) at (7,-1.5) {$\abs{y'}_{\s{D}} \circ F^{\times k} \circ (\abs{x_1'}_{\s{C}} \times \cdots \times \abs{x_k'}_{\s{C}})$};
		\node(31) at (0,-3) {$F \circ \abs{y}_{\s{C}} \circ (\abs{x_1}_{\s{C}} \times \cdots \times \abs{x_k}_{\s{C}})$};
		\node(32) at (7,-3) {$F \circ \abs{y'}_{\s{C}} \circ (\abs{x_1'}_{\s{C}} \times \cdots \times \abs{x_k'}_{\s{C}})$};
		
		\path[->]
		(11) edge node {} (12)
		(21) edge node {} (22)
		(31) edge node {} (32)
		(11) edge node {} (21)
		(21) edge node {} (31)
		(12) edge node {} (22)
		(22) edge node {} (32)
		;
	\end{tikzpicture}
	\]
where all vertical maps are induced by $\vd_\bullet$, and all horizontal maps are induced by the $f_1, \dots, f_k$ and $g$. The top square commutes by the $f_i$-naturality of the $\vd_{j_i}$, and the bottom square commutes by the $g$-naturality of $\vd_k$. Therefore $\vd_{j_1 + \cdots + j_k}$ is $\gamma(g;f_1,\dots,f_k)$-natural.
\end{proof}

\begin{prop}\label{prop:evbijfun} Let $\c{T} = (T_i)_{i \in I}$ be a set of exponents, where each $T_i$ is a finite, ordered $H_i$-set. The $2$-functor $\t{ev} : \c{SM}_{\c{T}}\t{-}\b{AlgLax} \to \c{T}\b{SMLax}$ is bijective on $1$-morphisms, and similarly in the pseudo and strict cases.
\end{prop}

\begin{proof}Let $\s{C}$ and $\s{D}$ be $\c{T}$-NSMCs, and let $\abs{\cdot}_{\s{C}} : \c{SM}_{\c{T}} \to \b{End}(\s{C})$ and $\abs{\cdot}_{\s{D}} : \c{SM}_{\c{T}} \to \b{End}(\s{D})$ be the corresponding $\c{SM}_{\c{T}}$-algebra structures.

Suppose $(F,\vd_\bullet) , (H , \epsilon_\bullet) : (\s{C},\abs{\cdot}_{\s{C}}) \rightrightarrows (\s{D},\abs{\cdot}_{\s{D}})$ are lax $\c{SM}_{\c{T}}$-algebra morphisms that evaluate to the same lax $\c{T}$-normed functor. Then $F = H$, and $(\vd_0)_{e} = (\epsilon_0)_{e}$, $(\vd_2)_{\otimes} = (\epsilon_2)_{\otimes}$, and $(\vd_{\abs{T_i}})_{\btn_{T_i}} = (\epsilon_{\abs{T_i}})_{\btn_{T_i}}$ for every $i \in I$. By conditions (ii) and (iv) of Definition \ref{defn:laxopmap}, the set of all $x \in \c{SM}_{\c{T}}$ such that the equation 
	\[
	(\vd_{\t{arity}(x)})_x = (\epsilon_{\t{arity}(x)})_x
	\]
holds is closed under operadic composition and the $G \times \Sigma$-action, and the operations $e()$, $\otimes(x_1,x_2)$, and $\btn_{T_i}(x_1,\dots,x_{\abs{T_i}})$ generate $\t{Ob}(\c{SM}_{\c{T}}) = F(S_{\c{T}})$. Therefore $(\vd_{\t{arity}(x)})_x = (\epsilon_{\t{arity}(x)})_x$ for all $x$, and hence $(F,\vd_\bullet) = (H , \epsilon_\bullet)$. Thus $\t{ev}$ is injective.

Now suppose $(F,F_\bullet) : \s{C} \to \s{D}$ is a lax $\c{T}$-normed functor. We must construct an extension of $(F,F_\bullet)$ to a lax $\c{SM}_{\c{T}}$-algebra morphism. Recall  Notation \ref{nota:eltFS} and the operad $\c{L}ax = \c{L}ax(\c{SM}_{\c{T}},\s{C},\s{D},F)$ from Definition \ref{defn:laxmapop}. The axioms for a lax $\c{T}$-normed functor imply that there is a map $S_{\c{T}} \to \c{L}ax$ of symmetric sequences, which sends $e$ to $(e(),F_e)$, $\otimes$ to $(\otimes(x_1,x_2),F_\otimes)$, and $\btn_{T_i}$ to $(\btn_{T_i}(x_1,\dots,x_{\abs{T}}),F_{\btn_{T_i}})$ for all $i \in I$. This freely extends to an operad map
	\[
	\Phi : \t{Ob}(\c{SM}_{\c{T}}) = F(S_{\c{T}}) \to \c{L}ax ,
	\]
and the composite $\pi_1 \circ \Phi : F(S_{\c{T}}) \to F(S_{\c{T}})$ fixes the generators of $F(S_{\c{T}})$. Therefore $\pi_1 \circ \Phi = \t{id}$, and we define $(\vd_n)_x$ by the equation
	\[
	\Phi_n(x) = \Big( x,(\vd_n)_x : \abs{x}_{\s{D}} \circ F^{\times n} \Rightarrow F \circ \abs{x}_{\s{C}} \Big).
	\]
By Proposition \ref{prop:laxopmap}, the data $(F , \vd_\bullet)$ satisfies axioms (ii)-(iv) of Definition \ref{defn:laxopmap}, so it remains to check (i), i.e. the naturality of $(\vd_n)_x$ in $x$. The axioms of a lax $\c{T}$-normed functor imply that $\vd_\bullet$ is natural for all of the morphisms $\alpha^{\pm 1}$, $\lambda^{\pm 1}$, $\rho^{\pm 1}$, $\beta$, and $\ups_T^{\pm 1}$ introduced in Notation \ref{nota:eltFS}. These generate every morphism in $\c{SM}_{\c{T}}$ under the $G$-category and operad structures, so by Lemmas \ref{lem:fnatGcat} and \ref{lem:fnatop} it follows that every transformation $(\vd_n)_x$ varies naturally in $x$. Thus $(F,\vd_\bullet) : ( \s{C} ,\abs{\cdot}_{\s{C}}) \to (\s{D} , \abs{\cdot}_{\s{D}})$ is a lax $\c{SM}_{\c{T}}$-algebra morphism, and it evaluates to $(F , F_\bullet)$ by construction.
\end{proof}

\begin{prop}\label{prop:evbijtrans} The $2$-functor $\t{ev} : \c{SM}_{\c{T}}\t{-}\b{AlgLax} \to \c{T}\b{SMLax}$ is bijective on $2$-morphisms, and similarly in the pseudo and strict cases.
\end{prop}

\begin{proof}The $2$-functor $\t{ev}$ does nothing to underlying $G$-natural transformations, and therefore $\t{ev}$ is injective on $2$-morphisms. 

Now for surjectivity. Suppose that $(F,F_\bullet) , (H, H_\bullet) : \s{C} \rightrightarrows \s{D}$ are lax $\c{T}$-normed functors and let $(F, \vd_\bullet), (H, \epsilon_\bullet) : (\s{C},\abs{\cdot}_{\s{C}}) \rightrightarrows (\s{D},\abs{\cdot}_{\s{D}})$ be the corresponding lax $\c{SM}_{\c{T}}$-algebra morphisms. Given any $\c{T}$-transformation $\omega : (F,F_\bullet) \Rightarrow (H, H_\bullet)$, we must show that $\omega$ is also a $\c{SM}_{\c{T}}$-transformation $\omega : (F, \vd_\bullet) \Rightarrow (H, \epsilon_\bullet)$. The set of all $x \in \c{SM}_{\c{T}}$ such that
	\[
	(\epsilon_{\t{arity}(x)})_x \bullet (\t{id}_{\abs{x}_{\s{D}}} \circ \omega^{\times \t{arity}(x)})
	= 
	(\omega \circ \t{id}_{\abs{x}_{\s{C}}}) \bullet (\vd_{\t{arity}(x)})_x
	\]
is closed under the $G \times \Sigma$-action and operadic composition. Here $\bullet$ denotes vertical composition and $\circ$ denotes horizontal composition. Since $\omega$ is a $\c{T}$-transformation, the equation above holds for the generators $e(), \otimes(x_1,x_2) , \btn_{T_i}(x_1,\dots,x_{\abs{T_i}}) \in \c{SM}_{\c{T}}$, and hence for all $x \in \c{SM}_{\c{T}}$.
\end{proof}

\section{NSMCs with strict relations}\label{sec:TRNSMC}

In this final section, we study NSMCs in which some coherence isomorphisms are identity maps, and we prove the analogue to Theorem \ref{thm:coh} for these structures (Theorem \ref{thm:cohR}). The key point is that the right adjoint $\widetilde{(-)}$ preserves quotients (Proposition \ref{prop:tilquot}). This fortuitous fact allows us to identify the operads that parametrize NSMCs with strict relations, and we conclude with a brief discussion of equivariant permutative categories.

\begin{defn}\label{defn:grrel} Suppose $X = (X_n)_{n \geq 0}$ is a sequence of sets. A \emph{graded binary relation on $X$} is a sequence $R = (R_n)_{n \geq 0}$ such that $R_n$ is a binary relation on the set $X_n$ for every $n \geq 0$. Given any $x$ and $y$, we shall sometimes write $x R y$ if there is some $n \geq 0$ such that $x,y \in X_n$ and $x R_n y$.
\end{defn}

\begin{defn}\label{defn:NRNSMC} Let $\c{T}$ be a set of exponents, and let $\s{C}$ be a $\c{T}$-NSMC whose $\c{SM}_{\c{T}}$-algebra structure is parametrized by the map
	\[
	\abs{\cdot}_{\s{C}} : \c{SM}_{\c{T}} \to \b{End}(\s{C}).
	\]
Suppose $R$ be a graded binary relation on $F(S_{\c{T}}) = \t{Ob}(\c{SM}_{\c{T}})$. Then $\s{C}$ is a \emph{$(\c{T},R)$-NSMC} if, for any $n \geq 0$ and $x,y \in F(S_{\c{T}})(n)$ such that $x R_n y$, we have
	\[
	\abs{x}_{\s{C}} = \abs{y}_{\s{C}}	\quad\t{and}\quad	\abs{x \to y}_{\s{C}} = \t{id}_{\abs{x}_{\s{C}}} = \t{id}_{\abs{y}_{\s{C}}}.
	\]
Let $(\c{T},R)\b{SMLax}$ be the full $2$-subcategory of $\c{T}\b{SMLax}$ whose objects are the $(\c{T},R)$-NSMCs, and similarly in the pseudo and strict cases.
\end{defn}

\begin{ex} A permutative category is a $(\c{T},R)$-NSMC for $G = *$, $\c{T} = \varnothing$, and
	\[
	R = \Big\{ \otimes(\otimes(x_1,x_2),x_3) \sim \otimes(x_1,\otimes(x_2,x_3)) \,\, , \,\, \otimes(e(),x_1) \sim x_1 \,\, , \,\, \otimes(x_1,e()) \sim x_1 \Big\}.
	\]
More generally, if $\s{C}$ is a permutative category and $J$ is a small right $G$-category, then the functor $G$-category $\b{Fun}(J,\s{C})$ considered in \S\ref{sec:FunNSMC} is a NSMC with strict relations. As explained previously, it has an external norm for every finite $H$-set $T$ such that $\bigcup_{j \in J} \t{Stab}_H(j) \subset \bigcap_{t \in T} \t{Stab}_H(t)$, and its levelwise symmetric monoidal structure is strictly associative and unital. There are other strict relations, but they are harder to pin down.
\end{ex}

We shall prove that $(\c{T},R)$-NSMC structures are parametrized by a quotient of $\c{SM}_{\c{T}}$, so we briefly recall the relevant formalism (cf. \cite[\S5.2]{RubComb}).

\begin{defn} Suppose $\s{O}$ is a discrete operad in $\b{Set}^G$, and $\sim$ is a graded binary relation on $\s{O}$. Then $\sim$ is a \emph{congruence relation} if
	\begin{enumerate}[label=(\roman*)]
		\item{}for each integer $n \geq 0$, $\sim_n$ is an equivalence relation on $\s{O}(n)$,
		\item{}for each integer $n \geq 0$, pair $(g, \sigma) \in G \times \Sigma_n$, and elements $x, x' \in \s{O}(n)$, if $x \sim_n x'$, then $g x \sigma \sim_n g x' \sigma$, and
		\item{}for all elements $y,y' \in \s{O}(k)$ and $x_i,x'_i \in \s{O}(j_i)$ for $i = 1, \dots , k$, if $y \sim_k y'$ and $x_i \sim_{j_i} x'_i$ for all $i$, then $\gamma(y;x_1, \dots , x_k) \sim_{\Sigma j_i} \gamma(y'; x'_1 , \dots , x'_k)$.
	\end{enumerate}
Every graded binary relation $R$ on $\s{O}$ generates a congruence relation $\la R \ra$, which is the intersection of all congruence relations containing $R$. The relation $\la R \ra$ is the smallest congruence relation that contains $R$.
\end{defn}

The usual arguments show that the levelwise quotient of an operad by a congruence relation is an operad with the expected universal property.

\begin{prop}If $R$ is a graded binary relation on an operad $\s{O}$ in $\b{Set}^G$, then the levelwise quotient $\s{O}/ \la R \ra = (\s{O}(n)/ \la R \ra_n )$ is an operad with:
	\begin{enumerate}
		\item{}identity element $[\t{id}]$, 
		\item{}$G \times \Sigma$-action $g[x]\sigma = [g x \sigma]$, and 
		\item{}composition $\gamma([y];[x_1],\dots,[x_k]) = [\gamma(y;x_1,\dots,x_k)]$.
	\end{enumerate}
The projection $\pi : \s{O} \to \s{O}/ \la R \ra$ is an operad map that sends $R$-related elements of $\s{O}$ to equal elements of $\s{O} / \la R \ra$, and it is initial with this property.
\end{prop}

Next, recall that $\til{(-)} : \b{Set} \to \b{Cat}$ denotes the right adjoint to the object functor $\t{Ob} : \b{Cat} \to \b{Set}$. It is not generally colimit preserving; for example, it does not send coproducts to coproducts. Nevertheless, it does preserve quotients of the form
	\[
	\pi : \s{O} \to \s{O}/\langle R \rangle,
	\]
when applied levelwise. This surprising feature of the functor $\til{(-)}$ is the key to proving the coherence theorem for $(\c{T},R)$-NSMCs.

\begin{prop}\label{prop:tilquot} Let $\s{O}$ be an operad in $\b{Set}^G$, let $R$ be a graded binary relation on $\s{O}$, and consider the operad map $\til{\pi} : \til{\s{O}} \to \til{\s{O}/\la R \ra}$. Then 
	\[
	\til{\pi}(x) = \til{\pi}(y)	\quad\t{and}\quad	\til{\pi}(x \to y) = \t{id}_{\til{\pi}(x)} = \t{id}_{\til{\pi}(y)}
	\]
whenever $x R y$, and the map $\til{\pi} : \til{\s{O}} \to \til{\s{O}/\la R \ra}$ is initial with this property.
\end{prop}

\begin{proof} All of the the claims about $\til{\pi}$ are clear except for the universal property. So suppose $\vp : \til{\s{O}} \to \s{Q}$ is an operad map in $\b{Op}(\b{Cat}^G)$ such that
	\begin{equation*}\tag{$*$}
	\vp(x) = \vp(y)	\quad\t{and}\quad	\vp(x \to y) = \t{id}_{\vp(x)} = \t{id}_{\vp(y)}
	\end{equation*}
whenever $x R y$. We must construct a unique factorization of $\vp$ through $\til{\pi}$. If such a map $\vp_* : \til{\s{O}/ \la R \ra} \to \s{Q}$ exists, then it is unique because $\til{\pi}$ is surjective. Thus, we need to show that $\vp_*$ exists. Define a graded relation $E$ on $\s{O}$ by
	\begin{center}
	$x E y$ if and only if condition $(*)$ holds.
	\end{center}
Then $E$ is a congruence relation on $\s{O}$, and $R \subset E$. Therefore $\la R \ra \subset E$, i.e. if $x \la R \ra y$, then $x E y$.

Now we construct the factorization of $\vp : \til{\s{O}} \to \s{Q}$ through $\til{\pi} : \til{\s{O}} \to \til{\s{O}/\la R \ra}$. Consider the map on objects
	\[
	\vp : \s{O} \to \t{Ob}(\s{Q})
	\]
first. If $x,y \in \s{O}$ and $x R y$, then $x E y$ and hence $\vp(x) = \vp(y)$. By the universal property of quotients, there is an induced operad map
	\[
	\vp_* : \s{O}/\la R \ra \to \t{Ob}(\s{Q})
	\]
defined by $\vp_*([x]) = \vp(x)$. Next, we extend $\vp_*$ to morphisms. Given any $[x] \to [y]$ in $\til{\s{O}/\la R \ra}$, choose representatives $x \in [x]$ and $y \in [y]$, and define
	\[
	\vp_*([x] \to [y]) = {\vp}(x \to y).
	\]
The value of ${\vp}(x \to y)$ does not depend on the choice. Indeed, if $x' \in [x]$ and $y' \in [y]$ are other representatives, then the below left commutative square in $\til{\s{O}}$ maps to the below right commutative square in $\s{Q}$
	\[
	\begin{tikzpicture}
		\node(A) at (0,0) {$x$};
		\node(B) at (0,-1.5) {$x'$};
		\node(C) at (2,0) {$y$};
		\node(D) at (2,-1.5) {$y'$};
		
		\path[->]
		(A) edge node {} (B)
		(A) edge node {} (C)
		(C) edge node {} (D)
		(B) edge node {} (D)
		;
		
		\node(A') at (4.5,0) {$\vp(x)$};
		\node(B') at (4.5,-1.5) {$\vp(x')$};
		\node(C') at (7,0) {$\vp(y)$};
		\node(D') at (7,-1.5) {$\vp(y')$};
		
		\path[->]
		(A') edge [left] node {$\t{id}$} (B')
		(A') edge [above] node {$\vp(x \to y)$} (C')
		(C') edge [right] node {$\t{id}$} (D')
		(B') edge [below] node {$\vp(x' \to y')$} (D')
		;
	\end{tikzpicture}
	\]
because $a \la R \ra b$ implies $a E b$, and hence $\vp(a \to b) = \t{id}$. It is straightforward to check that $\vp_* : \til{\s{O}/\la R \ra} \to \s{Q}$ is a map of operads in $\b{Cat}^G$, and the equation $\vp_* \circ \til{\pi} = \vp$ holds by construction.
\end{proof}

Thus, an $\til{\s{O}/\la R \ra}$-algebra structure on a $G$-category $\s{C}$ is precisely the same thing as an $\til{\s{O}}$-algebra structure on $\s{C}$ such that every morphism $x \to y$ between $R$-related operations $x,y \in \til{\s{O}}$ is an identity transformation in $\b{End}(\s{C})$.

Now specialize to the case $\s{O} = F(S_{\c{T}}) = \t{Ob}(\c{SM}_{\c{T}})$.

\begin{defn}\label{defn:SMNR} Suppose $\c{T} = (T_i)_{i \in I}$ is a set of exponents and $R$ is a graded binary relation on $F(S_{\c{T}})$  (cf. Definition \ref{defn:SMT}). We define
	\[
	\c{SM}_{(\c{T},R)} = \til{F(S_{\c{T}})/\la R \ra} ,
	\]
and we write $\til{\pi} : \c{SM}_{\c{T}} \to \c{SM}_{(\c{T},R)}$ for the quotient map. Recalling Notation \ref{nota:eltFS}, we write $[\otimes] , [e] , [\btn_{T_i}] , [\alpha] , [\lambda], \dots \in \c{SM}_{(\c{T},R)}$, for the images of the corresponding objects and morphisms in $\c{SM}_{\c{T}}$.
\end{defn}

We shall momentarily exploit the universal property of $\c{SM}_{\c{T}} \to \c{SM}_{(\c{T},R)}$ to identity the $2$-category of $\c{SM}_{(\c{T},R)}$-algebras with the $2$-category of $(\c{T},R)$-NSMCs, but first we consider pullback $2$-functors.

\begin{defn} Suppose $\vp = (\vp_n)_{n \geq 0} : \s{Q} \to \s{R}$ is a morphism of operads in $\b{Cat}^G$. The pullback $2$-functor
	\[
	\vp^* : \s{R}\t{-}\b{AlgLax} \to \s{Q}\t{-}\b{AlgLax}
	\]
is defined as follows:
	\begin{enumerate}
		\item{}For any $\s{R}$-algebra $\s{C}$ with structure map $\abs{\cdot}_{\s{C}} : \s{R} \to \b{End}(\s{C})$, we define $\vp^*(\s{C},\abs{\cdot}_{\s{C}}) = (\s{C} , \abs{\cdot}_{\s{C}} \circ \vp)$, i.e. $\abs{\cdot}_{\vp^*\s{C}} = \abs{\cdot}_{\s{C}} \circ \vp$.
		\item{}For any 1-morphism $(F,\vd_\bullet) : (\s{C},\abs{\cdot}_{\s{C}}) \to (\s{D},\abs{\cdot}_{\s{D}})$, we define $\vp^*(F,\vd_\bullet) = (F , (\vd_n \circ \vp_n)_{n \geq 0})$, i.e. $\vp^*\vd_n = \vd_n \circ \vp_n$.
		\item{}For any 2-morphism $\omega : (F,\vd_\bullet) \Rightarrow (F',\vd'_\bullet)$, we define $\vp^* \omega = \omega$.
	\end{enumerate}
The same definitions define $2$-functors in the pseudo and strict cases.
\end{defn}

Checking that these formulas define $2$-functors is a simple matter of unwinding the definitions. The following lemma is also a quick check.

\begin{lem}If $\s{C}$ is a $\c{SM}_{(\c{T},R)}$-algebra, then $\t{ev}(\til{\pi}^*\s{C})$ is a $(\c{T},R)$-NSMC.
\end{lem}

This justifies the next definition.

\begin{defn}\label{defn:evquot} Let $\c{T} = (T_i)_{i \in I}$ be a set of exponents and $R$ a graded binary relation on $F(S_{\c{T}})$. The evaluation $2$-functor
	\[
	\t{ev} : \c{SM}_{(\c{T},R)}\t{-}\b{AlgLax} \to (\c{T},R)\b{SMLax}
	\]
is the induced $2$-functor
	\[
	\begin{tikzpicture}
		\node(A) at (0,0) {$\c{SM}_{(\c{T},R)}\t{-}\b{AlgLax}$};
		\node(B) at (0,-1.5) {$\c{SM}_{\c{T}}\t{-}\b{AlgLax}$};
		\node(C) at (4,-1.5) {$\c{T}\b{SMLax}$};
		\node(D) at (4,0) {$(\c{T},R)\b{SMLax}$};
		\path[->]
		(A) edge [left] node {$\til{\pi}^*$} (B)
		(B) edge [below] node {$\t{ev}$} (C)
		(D) edge [right] node {$\t{inc}$} (C)
		;
		\path[->,dashed]
		(A) edge [above] node {$\t{ev}$} (D)
		;
	\end{tikzpicture}
	\]
and similarly in the pseudo and strict cases. The top evaluation $2$-functor sends a $\c{SM}_{(\c{T},R)}$-algebra $\abs{\cdot}_{\s{C}} : \c{SM}_{(\c{T},R)} \to \b{End}(\s{C})$ to the $(\c{T},R)$-NSMC
	\[
	\Big( \s{C} , \abs{[\otimes]}_{\s{C}} , \abs{[e]}_{\s{C}} , (\abs{[\sbtn_{T_i}]}_{\s{C}})_{i \in I} , \abs{[\alpha]}_{\s{C}}, \abs{[\lambda]}_{\s{C}} , \abs{[\rho]}_{\s{C}}, \abs{[\beta]}_{\s{C}} , (\abs{[\ups_{T_i}]}_{\s{C}})_{i \in I} \Big),
	\]
it sends a $\c{SM}_{(\c{T},R)}$-morphism $(F,\vd_\bullet)$ to the morphism
	\[
	\Big( F , (\vd_2)_{[\otimes]} , (\vd_0)_{[e]} , ((\vd_{\abs{T_i}})_{[\btn_{T_i}]})_{i \in I} \Big) ,
	\]
and it does nothing to $\c{SM}_{(\c{T},R)}$-transformations.
\end{defn}

Finally, we arrive at the coherence theorem for $(\c{T},R)$-NSMCs.

\begin{thm}\label{thm:cohR} Let $\c{T}$ be a set of exponents and $R$ a graded binary relation on $F(S_{\c{T}}) = \t{Ob}(\c{SM}_{\c{T}})$. There is a commutative triangle
	\[
	\begin{tikzpicture}
		\node(A) at (0,0) {$\c{SM}_{(\c{T},R)}\t{-}\b{AlgLax}$};
		\node(B) at (4,0) {$(\c{T},R)\b{SMLax}$};
		\node(C) at (2,-2) {$\ub{Cat}^G$};
		\path[->]
		(A) edge [above] node {$\t{ev}$} (B)
		(A) edge [below left] node {$\t{forget}$} (C)
		(B) edge [below right] node {$\t{forget}$} (C)
		;
	\end{tikzpicture}
	\]
of $2$-categories and $2$-functors, and the evaluation $2$-functor $\t{ev}$ is an isomorphism. Similarly in the strong and strict cases.
\end{thm}

\begin{proof} We check that $\t{ev}$ is bijective on objects, $1$-morphisms, and $2$-morphisms.  Injectivity in all three cases is clear, because $\til{\pi}^*$ and evaluation for $\c{T}$-NSMCs are both injective. We treat surjectivity on a case-by-case basis.

Suppose $\s{C}$ is a $(\c{T},R)$-NSMC. By Theorem \ref{thm:coh}, there is a corresponding $\c{SM}_{\c{T}}$-algebra structure $\abs{\cdot}_{\s{C}} : \c{SM}_{\c{T}} \to \b{End}(\s{C})$, and by definition, if $xRy$, then $\abs{x}_{\s{C}} = \abs{y}_{\s{C}}$ and $\abs{x \to y}_{\s{C}} = \t{id}$. Proposition \ref{prop:tilquot} applied to $\s{O} = F(S_{\c{T}})$ implies that $\abs{\cdot}_{\s{C}}$ factors through $\til{\pi}$ as a morphism $\abs{\cdot}_{\til{\pi}_*\s{C}} : \c{SM}_{(\c{T},R)} \to \b{End}(\s{C})$. Therefore $\t{ev}$ is surjective on objects.

Next, suppose $(F,F_\bullet) : \s{C} \to \s{D}$ is a $1$-morphism between $(\c{T},R)$-NSMCs. By Theorem \ref{thm:coh}, there is a $\c{SM}_{\c{T}}$-morphism $(F,\vd_\bullet) : (\s{C},\abs{\cdot}_{\s{C}}) \to (\s{D},\abs{\cdot}_{\s{D}})$ that evaluates to $(F,F_\bullet)$. Define a graded relation $\sim$ on $F(S_{\c{T}})$ by
	\begin{center}
		$x \sim y$ if and only if $\abs{x}_{\s{C}} = \abs{y}_{\s{C}}$, $\abs{x}_{\s{D}} = \abs{y}_{\s{D}}$, and $(\vd_{\t{arity}(x)})_x = (\vd_{\t{arity}(y)})_y$.
	\end{center}
Then $\sim$ is a congruence relation, and the naturality of $(\vd_n)_x$ in $x$ implies that $R \subset \,\, \sim$. Therefore $\la R \ra \subset \,\, \sim$, and hence the natural transformations
	\[
	(\ve_n)_{[x]} := (\vd_n)_x
	\]
are well-defined, where $[x]$ denotes the class of $x$ in $F(S_{\c{T}})/\la R \ra$. Straightforward checks show that $(F,\ve_\bullet) : \til{\pi}_*\s{C} \to \til{\pi}_*\s{D}$ is a $\c{SM}_{(\c{T},R)}$-morphism, and it evaluates to $(F,F_\bullet)$ by construction. Therefore $\t{ev}$ is surjective on $1$-morphisms.

Lastly, we consider $2$-morphisms. Suppose that $\omega : (F,F_\bullet) \Rightarrow (F',F'_\bullet)$ is a $\c{T}$-transformation between $1$-morphisms of $(\c{T},R)$-NSMCs. Then by Theorem \ref{thm:coh}, $\omega$ is also a transformation between the corresponding $\c{SM}_{\c{T}}$-morphisms $(F,\vd_\bullet)$ and $(F',\vd'_\bullet)$. It follows that $\omega$ is a $\c{SM}_{(\c{T},R)}$-transformation, because the map $\til{\pi} : \c{SM}_{\c{T}} \to \c{SM}_{(\c{T},R)}$ is surjective.
\end{proof}

In principle, this theorem identifies the algebras over a significant class of operads. Say that an operad $\s{O}$ in $\b{Cat}^G$ is \emph{codiscrete} if $\s{O} = \til{\t{Ob}(\s{O})}$.

\begin{cor} Suppose $\s{O}$ is a $\Sigma$-free codiscrete operad in $\b{Cat}^G$ such that $\s{O}(n)^G$ is nonempty for all $n \geq 0$. Then there is a set of exponents $\c{T}$ and a graded binary relation $R$ on $\t{Ob}(\c{SM}_{\c{T}})$ such that $\s{O}\t{-}\b{AlgLax} \cong (\c{T},R)\b{SMLax}$. Similarly in the strong and strict cases.
\end{cor}

\begin{proof}It will suffice to show that $\s{O} \cong \c{SM}_{(\c{T},R)}$ for some $\c{T}$ and $R$. Consider the object operad $\t{Ob}(\s{O})$. The counit $\ve : FU\t{Ob}(\s{O}) \to \t{Ob}(\s{O})$ of the free-forgetful adjunction $F : \b{Sym}(\b{Set}^G) \rightleftarrows \b{Op}(\b{Set}^G) : U$ is levelwise surjective, and if we define $xRy$ if and only if $\ve(x) = \ve(y)$, then $\ve$ induces an isomorphism $FU\t{Ob}(\s{O})/\la R \ra \cong \t{Ob}(\s{O})$. Now consider the symmetric sequence $U\t{Ob}(\s{O})$. Since $\s{O}$ is $\Sigma$-free and $\s{O}(n)^G \neq \varnothing$ for all $n$, we can choose a set of exponents $\c{T}$ such that $U\t{Ob}(\s{O}) \cong S_{\c{T}}$ (cf. Definition \ref{defn:SMT}). Therefore $\t{Ob}(\s{O}) \cong F(S_{\c{T}})/\la R \ra$ and
	\[
	\s{O} \cong \til{\t{Ob}(\s{O})} \cong \til{F(S_{\c{T}})/\la R \ra} = \c{SM}_{(\c{T},R)}. \qedhere
	\]
\end{proof}

Of course, explicitly determining the exponents and relations for a particular codiscrete operad $\s{O}$ is another matter. In view of the description of $F(S_{\c{T}})$ in \S\ref{subsec:SMTdesc}, identifying $\t{Ob}(\s{O})$ with $F(S_{\c{T}})/\la R \ra$ amounts to solving a word problem, which is often nontrivial. We find it easier to start with the desired generators and relations $(\c{T},R)$, and then to analyze the resulting operad $\c{SM}_{(\c{T},R)}$ as necessary.

\begin{ex}[Equivariant permutative categories] Theorem \ref{thm:cohR} identifies permutative categories with algebras over the operad
	\[
	\t{codisc}\Bigg( \frac{F(\Sigma_0/e \sqcup \Sigma_2/e)}
	{
	\Big\langle \!\!
	\otimes\!(\otimes(x_1,x_2),x_3) \sim \otimes(x_1,\otimes(x_2,x_3)) \,\, , \,\, \otimes(e(),x_1) \sim x_1 \sim \otimes(x_1,e())
	\Big\rangle
	}
	\Bigg).
	\]
This is precisely the permutativity operad $\s{P}$, so we have recovered the classical observation in \cite{MayPerm}.

The situation for equivariant permutative categories is somewhat more subtle. One approach, taken by Guillou and May \cite{GM}, starts from the operadic standpoint. Define the \emph{$G$-permutativity operad} by
	\[
	\s{P}_G = \b{Fun}(\til{G},\s{P}) \cong \til{\b{Set}(G,\b{As})}
	\]
and define a \emph{permutative $G$-category} to be a $\s{P}_G$-algebra. Such categories play a fundamental role in Guillou--May--Merling--Osorno's infinite loop space theory. Moreover, if $\s{C}$ is an ordinary permutative category, then $\b{Fun}(\til{G},\s{C}) \cong \b{Fun}(\bb{T}G,\s{C})$ is a permutative $G$-category, which furnishes a good supply of examples. 

Although the $G$-permutativity operad is easy to define and permutative $G$-categories have many desirable properties, there are some drawbacks to this approach. To start, the operad $\s{P}_G$ is not finitely generated \cite[Theorem 2.13]{BBKOOTX} and there is no known presentation of its object operad $\b{Set}(G,\b{As})$. These issues can be addressed in two special cases: when $G = C_2$ or $C_3$, there is an equivalent suboperad $\s{Q}_G \subset \s{P}_G$, whose object operad is finitely presented, and whose algebras can be understood in terms of Theorem \ref{thm:cohR} (loc. cit. Theorem 4.3). We know of no general solution to these problems, however. Another drawback to the operad $\s{P}_G$ is that it does not easily generalize to the $N_\infty$ setting. Not every $N_\infty$ operad is modeled by an operad of the form $\b{Fun}(J,\s{P})$ (cf. \cite[Counterexample 3.10]{RubComb}) or by a suboperad of $\s{P}_G$ (cf. \cite[Example B.2.1]{Bonventre}). These issues were sticking points in $N_\infty$ theory for some time.

We now present another approach to equivariant permutative categories, which is framed in terms of generators and relations, and which is compatible with $N_\infty$ theory. Suppose $\c{T} = (T_i)_{i \in I}$ is a set of exponents. We define a \emph{$\c{T}$-normed permutative category} ($\c{T}$-NPC) to be a $\c{T}$-NSMC such that the underlying symmetric monoidal structure is permutative, and the untwistors
	\[
	(\ups_{T_i})_{e,\dots,C, \dots, e}  \quad\t{and}\quad	(\ups_{T_i})_{e,\dots,e}
	\]
are identity maps for all $i \in I$ and $C \in \s{C}$. When $\abs{T_i} = 1$, we understand the left map to be $(\ups_{T_i})_C : \btn_{T_i}(C) \to C$, and when $\abs{T_i} = 0$, we understand the right map to be $\ups_{T_i} : \btn_{T_i}() \to e$. These conditions on untwistors ensure that there are no exotic nullary or unary operations.

The operad for $\c{T}$-NPCs is defined in parallel. Let $\b{As}_{\c{T}}$ denote the operad $F(S_{\c{T}})$ (cf. Definition \ref{defn:SMT}) modulo the relations
	\begin{align*}
		&\otimes(\otimes(x_1,x_2),x_3) \sim \otimes(x_1,\otimes(x_2,x_3))	\,\, , \,\, \otimes(e(),x_1) \sim x_1 \sim \otimes(x_1,e())	\\
		& \sbtn_{T_i}(e(),\dots,x_1,\dots,e()) \sim x_1	\\
		& \sbtn_{T_i}(e(),\dots,e()) \sim e()	
	\end{align*}
where $i$ ranges over all $i \in I$ in the second and third lines.	If $\abs{T_i} = 1$, then we understand the second line to be $\btn_{T_i}(x_1) \sim x_1$, and if $\abs{T_i} = 0$, then we understand the third line to be $\btn_{T_i}() \sim e()$. Define the $\c{T}$-permutativity operad to be $\s{P}_{\c{T}} = \til{\b{As}_{\c{T}}}$. Then Theorem \ref{thm:cohR} identifies $\c{T}$-NPCs with $\s{P}_{\c{T}}$-algebras.
\end{ex}

As with $\c{T}$-NSMCs, $\c{T}$-NPCs model $N_\infty$ spaces.

\begin{thm}\label{thm:NPCmodel} Let $\c{T}$ be a set of exponents and let $\s{C}$ be a $\c{T}$-NPC. Then the classifying space $B\s{C}$ is an algebra over the $N_\infty$ operad $B\s{P}_{\c{T}}$. Moreover, the class of admissible sets of $B\s{P}_{\c{T}}$ is the indexing system generated by $\c{T}$.
\end{thm}

\begin{proof} The operad $\b{As}_{\c{T}}$ is a $N$ operad whose class of admissible sets is the indexing system generated by $\c{T}$ (cf. \cite[Lemma 7.8]{RubComb}). From here, we can mimic the proof of Theorem \ref{thm:NSMCmodel}.
\end{proof}


\begin{thebibliography}{99}

\bibitem{BBKOOTX}
K. Bangs, S. Binegar, Y. Kim, K. Ormsby, A. Osorno, D. Tamas-Parris, and L. Xu.
Biased permutative equivariant categories.
To appear in Homology, Homotopy and Applications.

\bibitem{BH}
A. J. Blumberg and M. A. Hill.
Operadic multiplications in equivariant spectra, norms, and transfers.
Adv. Math. 285 (2015), 658--708.

\bibitem{Bohmann}
A. M. Bohmann.
A comparison of norm maps. With an appendix by Bohmann and Emily Riehl.
Proc. Amer. Math. Soc. 142 (2014), no. 4, 1413–1423.

\bibitem{Bonventre}
P. Bonventre.
Comparison of Models for Equivariant Operads.
Thesis (Ph.D.)--University of Virginia. 2017. 243 pp.

\bibitem{GM}
B. Guillou and J. P. May.
Equivariant iterated loop space theory and permutative $G$-categories.
Algebr. Geom. Topol. 17 (2017), no. 6, 3259--3339.

\bibitem{GMM}
B. Guillou, J. P. May, and M. Merling.
Categorical models for equivariant classifying spaces.
Algebr. Geom. Topol. 17 (2017), no. 5, 2565--2602.

\bibitem{GMMO}
B. Guillou, J. P. May, M. Merling, and A. M. Osorno.
Symmetric monoidal $G$-categories and their strictification.
Q. J. Math. 71 (2020), no. 1, 207--246.

\bibitem{HH}
M. A. Hill and M. J. Hopkins.
Equivariant symmetric monoidal structures.
Preprint. arXiv:1610.03114.

\bibitem{HHR}
M. A. Hill, M. J. Hopkins, and D. C. Ravenel.
On the nonexistence of elements of Kervaire invariant one.
Ann. of Math. (2) 184 (2016), no. 1, 1–262.

\bibitem{Kelly}
G. M. Kelly.
On MacLane's conditions for coherence of natural associativities, commutativities, etc.
J. Algebra {1} 1964 397--402.

\bibitem{MacLane}
S. Mac Lane.
Natural associativity and commutativity.
Rice Univ. Studies {49} 1963 no. 4, 28--46.

\bibitem{CWM}
S. Mac Lane.
Categories for the working mathematician.
Second edition.
Graduate Texts in Mathematics, 5. Springer-Verlag, New York, 1998. xii+314 pp. 

\bibitem{MayGILS}
J. P. May.
The geometry of iterated loop spaces.
Lectures Notes in Mathematics, Vol. 271. Springer-Verlag, Berlin-New York, 1972. viii+175 pp.

\bibitem{MayPerm}
J. P. May.
$E_\infty$  spaces, group completions, and permutative categories. New developments in topology (Proc. Sympos. Algebraic Topology, Oxford, 1972), pp. 61–93. London Math. Soc. Lecture Note Ser., No. 11, Cambridge Univ. Press, London, 1974.

\bibitem{RubComb}
J. Rubin.
Combinatorial $N_\infty$ operads.
Preprint. arXiv:1705.03585.

\bibitem{Segalclass}
G. Segal.
Classifying spaces and spectral sequences.
Inst. Hautes \'{E}tudes Sci. Publ. Math. No. 34 (1968), 105--112.

\bibitem{Segal}
G. Segal.
Categories and cohomology theories.
Topology 13 (1974), 293--312.

\end{thebibliography}
\end{document}